%% file: FW-R2_v4-rmf.tex
\documentclass[11pt]{article}

\usepackage{graphicx}
\usepackage[colorlinks=true,breaklinks=true,bookmarks=true,urlcolor=blue,
     citecolor=blue,linkcolor=blue,bookmarksopen=false,draft=false]{hyperref}
\input{preamble2}
\usepackage{fullpage}
\numberwithin{equation}{section}
\numberwithin{lemma}{section}
\numberwithin{theorem}{section}
\numberwithin{prop}{section}
\numberwithin{remark}{section}
\newcommand{\poi}{(P)}
\newcommand{\doi}{(D)}

\newcommand{\EE}{\mathbb{E}}

\newcommand{\rh}{R_h}
\newcommand{\TV}{{\rm TV}}

\begin{document}

\title{Analysis of the Frank-Wolfe Method for Convex Composite Optimization involving a Logarithmically-Homogeneous Barrier\thanks{Research
supported by AFOSR Grant No. FA9550-19-1-0240.
}
}


\author{Renbo Zhao\thanks{MIT Operations Research Center, 77 Massachusetts Avenue, Cambridge, MA   02139 ({mailto:  renboz@mit.edu}).}         \and
        Robert M. Freund\thanks{MIT Sloan School of Management, 77 Massachusetts Avenue, Cambridge, MA   02139 ({mailto:  rfreund@mit.edu}).} 
}


%

\maketitle

\begin{abstract} 

{We present and analyze a new generalized Frank-Wolfe method for the composite optimization problem $(P):  {\min}_{x\in\bbR^n} \; f(\mathsf{A} x) + h(x)$, where $f$ is a $\theta$-logarithmically-homogeneous self-concordant barrier, $\rvA$ is a linear operator and the function $h$ has bounded domain but is possibly non-smooth.  We show that our generalized Frank-Wolfe method requires $O((\delta_0 + \theta + R_h)\ln(\delta_0) + (\theta + R_h)^2/\varepsilon)$ iterations to produce an $\varepsilon$-approximate solution, where $\delta_0$ denotes the initial optimality gap and $R_h$ is the variation of $h$ on its domain. This result establishes certain intrinsic connections between $\theta$-logarithmically homogeneous barriers and the Frank-Wolfe method. When specialized to the $D$-optimal design problem, we essentially recover the complexity obtained by Khachiyan~\cite{khachiyan1996rounding} using  the Frank-Wolfe method with exact line-search.  We also study the (Fenchel) dual problem of $(P)$, and we show that our new method is equivalent to an adaptive-step-size mirror descent method applied to the dual problem. This enables us to provide iteration complexity bounds for the mirror descent method despite that the dual objective function is non-Lipschitz  and has unbounded domain. In addition, we present computational experiments that point to the potential usefulness of our generalized Frank-Wolfe method on Poisson image de-blurring problems with TV regularization, and on simulated PET problem instances.}    

\end{abstract}

\vspace{0.2in}

\section{Introduction}
\label{intro}

%


We present and analyze a new generalized Frank-Wolfe method~\cite{Frank_56,dem1967minimization,Lev_66,Canon_68,Dunn1978,Dunn1979,Dunn1980} for the following composite optimization problem:
\begin{equation}
(P): \ \ \ F^*:= {\min}_{x\in\bbR^n} \;[F(x):= f(\rvA x) + h(x)] \label{poi}
\end{equation}
where $f: \bbR^m\to\bbR\cup\{+\infty\}$ is an extended real-valued convex function that is differentiable on its domain $\dom f:=\{u \in \mathbb{R}^m : f(u) < +\infty \}$, $\rvA:\bbR^{n}\to \bbR^{m}$ is a linear operator (though not necessarily invertible or surjective) and the function $h:\bbR^n\to\bbR\cup\{+\infty\}$ is proper, closed and convex (but possibly non-smooth), for which $\dom h$ is a nonempty compact convex set.  Furthermore, and in contrast to the standard setting where $f$ is assumed to be $L$-smooth on $\dom F$ (i.e., its gradient is $L$-Lipschitz on $\dom F$), our focus is on the setting where $f$ belongs to a particularly special and important class of functions that arise in practice, namely $\theta$-logarithmically-homogeneous self-concordant barrier functions (whose definition and properties will be reviewed below).  For convenience we distinguish between $u \mapsto f(u)$ and $x \mapsto f(\rvA x)$ by defining 
\begin{equation}\label{stones} \bar{f}(x):=f(\rvA x) \ . \end{equation}

The Frank-Wolfe method was developed in 1956 in the seminal paper of Frank and Wolfe~\cite{Frank_56} for the case $h = \iota_{\calX}$, where $\iota_{\calX}$ denotes the indicator function of $\calX$ (i.e., $\iota_{\calX}(x) = 0$ for $x\in\calX$ and $+\infty$ otherwise), and $\calX$ is a (bounded) polytope. (In particular, $\poi$ then is the constrained problem $\min_{x\in\calX}\, \bar f(x)$.) The Frank-Wolfe method was a significant focus of research up through approximately 1980, during which time it was generalized to handle more general compact sets $\calX$, see e.g., \cite{dem1967minimization,Dunn1978,Dunn1979,Dunn1980}. Each iteration of the Frank-Wolfe method computes the minimizer of the first-order (gradient) approximation of $\bar{f}(x)$ on $\calX$, and constructs the next iterate by moving towards the minimizer.  Just in the last decade, and due to the importance of optimization in modern computational statistics and machine learning, the Frank-Wolfe method has again attracted significant interest (see \cite{Jaggi_13,Harcha_15,Freund_16,Freund_17,Nest_18,Ghad_19} among many others), for at least two reasons. First, in many modern application settings, $\calX$ is a computationally ``simple'' set for which the linear optimization sub-problem in the Frank-Wolfe method is easier to solve compared to the standard projection sub-problems required by other first-order methods, see e.g.,~\cite{Nest_13}. Secondly, the Frank-Wolfe method naturally produces ``structured'' (such as sparse, or low-rank) solutions in several important settings, which is very useful in the high-dimensional regime in machine learning. This is because each iterate of the Frank-Wolfe method is a convex combination of all the previous linear optimization sub-problem solutions, and hence if the extreme points of $\calX$ are unit coordinate vectors (for example), then the $k$-th iterate will have at most $k$ non-zeroes. This statement can be made more precise in specific settings, and also generalizes to the matrix setting where $\calX$ is the nuclear norm ball~\cite{Harcha_15} --- in this setting the $k$-th (matrix) iterate will have rank at most $k$.

More recently, the Frank-Wolfe method has been generalized to the composite setting, where the function $h$ is a general convex non-smooth function with compact domain $\calX$, see e.g.,~\cite{Bach_15,Nest_18,Ghad_19}. In this generalized framework, the sub-problem solved at iteration $k$ is 
\begin{equation}\label{subcbday}
\minimize_{x \in \bbR^n}\; \langle \nabla \bar{f}(x^k), x \rangle + h(x) \ ,
\end{equation} which specializes to the standard Frank-Wolfe sub-problem in the case when $h = \iota_\calX$.
In certain situations, this minimization problem admits (relatively) easily computable solutions despite the presence of the non-smooth function $h$. For example, if $h= \barh + \iota_\calP$, where $\barh$ is a polyhedral function and $\calP$ is a polytope, then \eqref{subcbday} can be reformulated as a linear optimization problem (LP), which can be solved efficiently if it has moderate size or a special structure, e.g., network flow structure~\cite{Harcha_15}. For more such examples we refer the reader to~\cite{Nest_18}.  

In addition, there has been recent research work on using the Frank-Wolfe method to solve the projection sub-problems (which are constrained quadratic problems) that arise in various optimization algorithms. For example,  \cite{Liu_20} presents a  projected Newton method for solving a class of problems that is somewhat different from (but related to) \eqref{poi}; specifically \cite{Liu_20} assumes that the linear operator $\rvA$ is invertible and the function $f$ is self-concordant but is not necessarily a logarithmically-homogeneous barrier. The Frank-Wolfe method is used therein to solve each projection sub-problem in the projected Newton method, and \cite{Liu_20} shows that the total number of linear minimization sub-problems needed is $O(\varepsilon^{-(1+o(1))})$.  Another such example is in \cite[Section~5]{Doikov_20}, which develops an affine-invariant trust-region type of method for solving a class of convex composite optimization problems in a similar form as~\eqref{poi}, with the key difference being that in \cite{Doikov_20} $f$ is assumed to be twice differentiable with Lipschitz Hessian on $\dom h$. The Frank-Wolfe method is used in \cite{Doikov_20} to solve each projection sub-problem, wherein it is shown that the total number of linear minimization sub-problems  needed is $O(\varepsilon^{-1})$.

When analyzing the convergence of the standard or generalized Frank-Wolfe method, almost all such analyses rely on the $L$-smooth assumption of the function $f$. 
Perhaps accidentally, the first specific attempt to extend the Frank-Wolfe method and analysis beyond the case of $L$-smooth functions is due to Khachiyan \cite{khachiyan1996rounding}.  In the specific case of the $D$-optimal design problem \cite{Fedorov_72}, Khachiyan \cite{khachiyan1996rounding} developed a ``barycentric coordinate descent'' method with an elegant computational complexity analysis, and it turns out that this method is none other than the Frank-Wolfe method with exact line-search~\cite{sun2004computation,alperman}. Khachiyan's proof of his complexity result (essentially $O(n^2/\varepsilon)$ iterations) used clever arguments that do not easily carry over elsewhere, and hence begged the question of whether or not any of the arguments in \cite{khachiyan1996rounding} might underly any general themes beyond $D$-optimal design, and if so what might be the mathematical structures driving any such themes?  In this work, we provide affirmative answers to these questions, by considering the $D$-optimal design problem as a special instance of the broader class of composite optimization problem $(P)$.  The second attempt was the recent paper of Dvurechensky et al.\ \cite{Dvu_20}, which presented and analyzed an adaptive step-size Frank-Wolfe method for tackling the problem $\min_{x \in \calX} \bar {f}(x)$ where $\bar f$ is assumed to be a non-degenerate (i.e., with positive-definite Hessians) self-concordant function and $\calX$ is a compact convex set, and was the first paper to study the Frank-Wolfe method for these special functions. 
The set-up in \cite{Dvu_20} can be seen as an instance of $\poi$ with $h$ being the indicator function of $\calX$, namely, $h=\iota_\calX$, and the additional assumption that $\bar f$ is non-degenerate, which we do not need.  (Note that in our setting, this amounts to assuming that $\rvA = \rvI$, namely the identity operator or that the linear operator $\rvA$ is invertible.) 
However, unlike \cite{Dvu_20}, we additionally assume that $f$ is $\theta$-logarithmically homogeneous.  As our analysis will show, this last property --- which holds true for all applications that we are aware of --- is the key property that leads to relatively simple and natural computational guarantees for the Frank-Wolfe method in this expanded relevant setting.

Let us now review the formal definition of a $\theta$-logarithmically-homogeneous self-concordant barrier function. Let $\calK\subsetneqq \bbR^m$ be a regular cone, i.e., $\calK$ is closed, convex, pointed ($\calK$ contains no line), and has nonempty interior ($\inter\calK\ne \emptyset$).  We say that $f$ is a $\theta$-logarithmically-homogeneous (non-degenerate) self-concordant barrier on $\calK$ for some $\theta \ge 1$ and we write ``$f \in \calB_\theta(\calK)$'', if $f$ is three-times differentiable and strictly convex on $\inter \calK$ and satisfies the following three properties: 
\begin{enumerate}[label=(P\arabic*),leftmargin=3.9em]
\item $\abst{D^3f(u)[w,w,w]}\le 2(\lranglet{H(u)w}{w})^{3/2}$ \ $\forall\,u\in\inter\calK$, $\forall\,w\in\bbR^m$,\label{item:third_second_bounded}
\item $f(u_k)\to  \infty$ for any $\{u_k\}_{k\ge 1}\subseteq\inter\calK$ such that $u_k\to u\in\bdry\calK$, and\label{item:boundary_growth}
\item $f(tu) = f(u) - \theta\ln (t)$ \  $\forall\,u\in\inter\calK$, $\forall\,t>0$ , \label{item:log_homogeneous}
\end{enumerate}
where $H(u)$ denotes the Hessian of $f$ at $u\in\inter\calK$. For details on these properties, we refer readers to Nesterov and Nemirovski~\cite[Section 2.3.3]{Nest_94} and Renegar \cite[Section 2.3.5]{Renegar_01}.  
 Properties (P1) and (P2) correspond to $f$ being a (standard, strongly) self-concordant function on $\inter\calK$ (cf.~\cite[Remark~2.1.1]{Nest_94}), and property (P3) corresponds to $f$ being a $\theta$-logarithmically-homogeneous barrier function on $\calK$.  Here $\theta$ is called the {\em complexity parameter} of $f$ in the terminology of Renegar~\cite{Renegar_01}.  The two prototypical examples of such functions are (i) $-\ln\det(U)$ for $U \in \calK := \mathbb{S}_{+}^{k}$ and $\theta = k$, and (ii) $-\sum_{j=1}^m w_j \ln(u_j)$ for $u \in \calK := \mathbb{R}_{+}^{m}$ and $\theta = \sum_{j=1}^m w_j$ where $w_1, \ldots, w_n \ge 1$, see \cite{Nest_94, Renegar_01}.  

We now present some application examples of $\poi$ where $f \in \calB_\theta(\calK)$, including the aforementioned $D$-optimal design problem.

\subsection{Applications}\label{sec:applications}

\vspace{2ex}
\noindent {\em 1.\ Poisson image de-blurring with total variation (TV) regularization} \cite{Harmany_12,Dey_06,Chambolle_18}. 
Let the $m\times n$ matrix $X$ be the true representation of an image, such that each entry $X_{ij}\ge 0$ represents the intensity of the pixel at location $(i,j)\in[m]\times[n]$, 
and $X_{ij}\in\{0,1,\ldots,M\}$, where $M:=2^b-1$ for $b$-bit images. 
In many applications, ranging from microscopy to astronomy, we observe a blurred image contaminated by Poisson noise, which we denote by $Y$,  and we wish to estimate the true image $X$ from $Y$. 
The generative model of $Y$ from $X$ is presumed to be as follows. Let $\rvA:\bbR^{m\times n}\to \bbR^{m\times n}$ denote the 2D discrete convolutional (linear) operator with periodic boundary conditions, which is assumed to be known. This convolutional operator is defined by a $p\times p$ 2D convolutional kernel with a size 
$q:=p^2$ that is typically much smaller than the size of image $N:= mn$. (For an illustration of the 2D convolution, see~\cite{Bhar_18} for example.) The blurred image $\tilY$ is obtained by passing $X$ through $\rvA$, i.e., $\tilY := \rvA(X)$, and the observed image $Y$ results from adding independent entry-wise Poisson noise to $\tilY$, i.e., $ Y_{ij}\sim {\sf Poiss}(\tilY_{ij})$, for all  $(i,j)\in[m]\times[n]$, and 
$\{Y_{ij}\}_{(i,j)\in[m]\times [n]}$ are assumed to be independent. 

It will be preferable to work with vectors in addition to matrices, whereby we equivalently describe the above model using vector notation as follows. We denote $X  = [x_1 \;\cdots\; x_m]^\top$, where $x_i^\top$ denotes the $i$-th row of $X$, and let ${\sf vec}:\bbR^{m\times n}\to \bbR^{mn}$ denote the vectorization operator that sequentially concatenates $X$ into the column vector $ {\sf vec}(X) := x :=[x_1^\top \cdots~ x_m^\top]^\top$, and let ${\sf mat}(x)$ denote the inverse operator of ${\sf vec}$, so that ${\sf mat}(x) = X$. Define $y := {\sf vec}(Y)$ and $\tily := {\sf vec}(\tilY)$.  
In addition, we represent $\rvA$ in its matrix form $A\in\bbR^{N\times N}$ (recall $N := m  n$), such that 
$\tily := A x$. Furthermore, let us represent $A:= [a_1\;\ldots\;a_N]^\top$, where $a_l^\top$ denotes the $l$-th row of $A$ for $l \in [N]$.  Note that $A$ is a sparse doubly-block-circulant matrix, such that each row $a_l^\top$ of $A$ has at most $q$ non-zeros, where $q\ll N$ denotes the size of the 2D convolution kernel. 
Finally, we have $y_{l}\sim {\sf Poiss}(\tily_{l})$ for all  $l\in[N]$, and $\{y_l\}_{l\in[N]}$ are independent.  

 The maximum likelihood (ML) estimator of $X$ from the observed image $Y$ is the optimal solution of the following optimization problem:
\begin{align}
{\min}_{x\in\bbR^N}&\;\;  -\textstyle\sum_{l=1}^{N} y_l\ln(a_l^\top x) + (\sum_{l=1}^{N} a_l)^\top x \quad 
\st\;\; 0\le  x\le Me \ ,   \label{eq:deblurring}
\end{align} 
where $e$ denotes the vector with all entries equal to one. 
In addition, following~\cite{Rudin_92}, in order to recover a smooth image with sharp edges, 
we add Total Variation (TV) regularization to the objective function in~\eqref{eq:deblurring}, which yields the following regularized ML estimation problem:
\begin{align}
{\min}_{x\in\bbR^N}&\;\; \barF(x):= -\textstyle\sum_{l=1}^{N} y_l\ln(a_l^\top x) + (\sum_{l=1}^{N} a_l)^\top x + \lambda {\rm TV}(x)\quad \st\;\; 0\le  x\le Me \ , 
\label{eq:deblurring_TV}
\end{align} 
where 
\begin{align*}{\rm TV}(x)&:= \textstyle\sum_{i=1}^m \sum_{j=1}^{n-1} \abs{[{\sf mat}(x)]_{i,j} - [{\sf mat}(x)]_{i,j+1}} + \textstyle\sum_{i=1}^{m-1} \sum_{j=1}^{n} \abs{[{\sf mat}(x)]_{i,j} - [{\sf mat}(x)]_{i+1,j}} \\ &:= \textstyle\sum_{i=1}^m \sum_{j=1}^{n-1} \abs{X_{i,j} - X_{i,j+1}} + \textstyle\sum_{i=1}^{m-1} \sum_{j=1}^{n} \abs{X_{i,j} - X_{i+1,j}}  
\end{align*} 
is a standard formulation of the total variation.  Here we see that \eqref{eq:deblurring_TV} is an instance of $\poi$ with $f(u) := -\textstyle\sum_{l=1}^N y_l\ln \big(u_l)$, $\calK:= \bbR^N_+$, $h(x) := (\sum_{l=1}^{N} a_l)^\top x + \lambda {\rm TV}(x) + \iota_{\calX}$ where $\calX = \{ x \in \bbR^N : 0 \le x \le Me \}$, $\rvA$ is defined by $(\rvA x)_l := a_l^\top x$, $l=1, \ldots, N$, and $\theta = \sum_{l=1}^N y_l$.  We note that $y_l \ge 1$ whenever $y_l \ne 0$ for all $l \in [N]$, and hence $f \in \calB_\theta(\calK)$.  In Section \ref{sec:delurring} we will discuss how the Frank-Wolfe sub-problem \eqref{subcbday} associated with 
\eqref{eq:deblurring_TV} can be efficiently solved.

\vspace{2ex}
\noindent {\em 2.\ Positron emission tomography (PET)}~\cite{Shepp_82,BenTal_01}. PET is a medical imaging technique that measures the metabolic activities of human tissues and organs. Typically, radioactive materials are injected into the organ of interest, and these materials emit (radioactive) events 
that can be detected by PET scanners.  The mathematical model behind this process is described as follows. Suppose that an emission object (e.g., a human organ) has been discretized into $n$ voxels. The number of 
events emitted by voxel $i$ ($i\in[n]$) is a Poisson random variable $\tilde X_i$ with {\em unknown} mean $x_i \ge 0$ and so $\tilde X_i\sim {\sf Poiss}(x_i)$, and  furthermore $\{\tilde X_i\}_{i=1}^n$ are assumed to be independent.  We also have a scanner array with $m$ bins. Each event emitted by voxel $i$ has a {\em known} probability $p_{ij}$ of being detected by bin $j$ ($j\in[m]$), and we assume that $\sum_{j=1}^m p_{ij} = 1$, i.e., the event will be detected by exactly one bin.  Let $\tilde Y_j$ denote the total number of events detected by bin $j$, whereby
\begin{equation}\label{thursday01}
\mathbb{E} [\tilde Y_j ]:= y_j := \textstyle\sum_{i=1}^n p_{ij} x_i \ . 
\end{equation}
By Poisson thinning and superposition, it follows that $\{\tilde Y_j\}_{j=1}^m$ are independent random variables 
and $\tilde Y_j\sim {\sf Poiss}(y_j)$ for all $j\in[m]$.  

We seek to perform maximum-likelihood (ML) estimation of the unknown means $\{x_i\}_{i=1}^n$ based on observations $\{Y_j\}_{j=1}^m$of the random variables $\{\tilde Y_j\}_{j=1}^m$. From the model above, we easily see that the log-likelihood of observing $\{Y_j\}_{j=1}^m$ given $\{\tilde X_i\}_{i=1}^n$ is (up to some constants)
\begin{equation}
l(x):= - \textstyle\sum_{i=1}^n x_i + \textstyle\sum_{j=1}^mY_j\ln \big(\sum_{i=1}^n p_{ij} x_i\big) \ ,
\end{equation}
and therefore an ML estimate of $\{x_i\}_{i=1}^n$ is given by an optimal solution $x^*$ of 
\begin{equation}
{\max}_{x\ge 0} \;l(x) \ . \label{eq:PET_0}
\end{equation}
  %
It follows from the first-order optimality conditions that 
any optimal solution $x$ must satisfy 
\begin{equation}
\textstyle\sum_{i=1}^n x_i = S := \sum_{j=1}^m Y_j \ ,  \label{eq:PET_extra_constraint}
\end{equation}
and by incorporating \eqref{eq:PET_extra_constraint} into \eqref{eq:PET_0} and defining the re-scaled variable $z:= x/S$, \eqref{eq:PET_0} can be equivalently written as 
\begin{equation}
{\min}_z \; L(z) := -\textstyle\sum_{j=1}^m Y_j\ln \big(\sum_{i=1}^n p_{ij} z_i\big)\quad \st\;\; {z\in\Delta_n} \ , \label{eq:PET_final}
\end{equation} where $\Delta_n := \{ z \in \mathbb{R}^n : \sum_{i=1}^n z_i =1, \ z \ge 0 \}$ is the unit simplex in $\mathbb{R}^n$.  Here we see that \eqref{eq:PET_final} is an instance of \eqref{poi} with $f(u) := -\textstyle\sum_{j=1}^m Y_j\ln \big(u_j)$, $\calK:= \bbR^m_+$, $h := \iota_{\Delta_n}$, $\rvA$ defined by $(\rvA z)_j := \sum_{i=1}^n p_{ij} z_i$, $j=1, \ldots, m$, and $\theta = \sum_{j=1}^m Y_j$.  We note that $Y_j \ge 1$ whenever $Y_j \ne 0$ for all $j \in [m]$, and hence $f \in \calB_\theta(\calK)$. 

\vspace{1ex}

\noindent {\em 3. Poisson phase retrieval}~\cite{Odor_16}. In Poisson phase retrieval, we seek to estimate an unknown unit complex signal $x\in\bbC^n$, namely $\normt{x}_2:= (x^H x)^{1/2}=1$ where $x^H$ denotes the conjugate transpose of $x$.  We estimate $x$ using $m$ linear measurements that are subject to Poisson noise; for $j\in[m]$, the $j$-th measurement vector is denoted by $a_j\in \bbC^n$, and the measurement outcome $y_j$ is a Poisson random variable such that $y_j\sim {\sf Poiss}(\tily_j)$, where $\tily_j:= \abst{\ipt{a_j}{x}}^2$.  Oder et al.~\cite{Odor_16} proposed to estimate $x$ by solving the following matrix optimization problem:
\begin{equation}
{\min}_X -\textstyle\sum_{j=1}^m y_j\ln \ipt{a_j a_j^H}{X} + \ipt{\sum_{j=1}^m a_j a_j^H}{X}\quad \st\;\; {X\in\calX:= \{X\in\bbH_+^n:\tr(X)\le c\}} \ , \label{eq:PFR}
\end{equation}
where $\ipt{\cdot}{\cdot}$ denotes the Frobenius matrix inner product, 
$\bbH_+^n$ denotes the set of complex Hermitian positive semi-definite matrices of order $n$, $\tr(X)$ denotes the trace of $X$, and the parameter $c>0$ is typically chosen as $c = (1/m)\sum_{j=1}^m y_j$. Let $X^*$ be the optimal solution of~\eqref{eq:PFR}. One then computes a unit eigenvector $\bar x$ associated with the largest eigenvalue of $X^*$ and uses $\bar x$ as the estimate of $x$. Note that~\eqref{eq:PFR} has a similar form to~\eqref{eq:PET_final} except in two ways:  first, the objective function in \eqref{eq:PFR} has an additional linear term, and second, the constraint set is the intersection of a nuclear norm ball with the positive semi-definite cone, 
instead of a simplex. Therefore, using the same arguments as above, we see that~\eqref{eq:PFR} is an instance of \eqref{poi}.  To solve~\eqref{eq:PFR}, \cite{Odor_16}  proposed a Frank-Wolfe method with a pre-determined step-size sequence, and showed that this method converges with rate $O(C/k)$, where $C$ depends on several factors including (i) the diameter of $\calX$ under the spectral norm, (ii) $\max_{j\in[m]} \normt{a_j}^2_2$, (iii) $\max_{j\in[m]}\max_{X\in\calX} \ipt{a_j a_j^H}{X}$, and (iv) $\min_{j\in[m]} \ipt{a_j a_j^H}{X_0}$ where $X_0\in\calX$ denotes the starting point of the Frank-Wolfe method.

\vspace{1ex}
\noindent {\em 4.\ Optimal expected log investment}~\cite{Cover_84,Algoet_88,Vardi_93}. In this problem we consider $n$ stocks in the market, and let  $R_i$ denote the random per-unit capital return on investing in stock $i$, for $i\in[n]$.  The random vector $R:= (R_1,\ldots,R_n)$ has unknown distribution $P$. An investor allocates her investment capital over these $n$ stocks, and let $w_i$ denote the (nonnegative) proportion of capital invested in stock $i$, whereby $w_i\ge 0$ for all $i\in[n]$ and $\sum_{i=1}^n w_i=1$. Define $w:=(w_1,\ldots,w_n)$. The goal of the investor is to maximize her expected log return $f(w):= \bbE_{R\sim P}[\ln(w^\top R)]$ subject to the constraint $w\in\Delta_n$ where $\Delta_n :=\{w \in \bbR^n : w \ge 0, \ e^Tw = 1\}$. The naturalness of this objective can be justified from several perspectives involving the principle that money compounds multiplicatively rather than additively, see the discussion and references in \cite{Cover_84,Algoet_88}.  Since $P$ is unknown, one can use a (historical) data-driven empirical distribution such as $\hatP_m:= \sum_{j=1}^m p_j \delta_{r_j}$, where $p_j > 0$, $\sum_{j=1}^m p_j=1$, $r_j\in\bbR^n$ is a realization of $R$ and $\delta_{r_j}$ denotes the unit point mass at  $r_j$ for $j \in [m]$. Under this empirical distribution, the investor instead solves the problem:
\begin{align}
{\min}_{w\in\Delta_n}&\;  -\textstyle\sum_{j=1}^{m} p_j\ln(r_j^\top w) \ . \label{eq:log_invest} 
\end{align}
Note that~\eqref{eq:log_invest} has the same basic format as the PET problem in~\eqref{eq:PET_final}. Indeed, 
both of these problems fall under a more general class of problems called ``positive linear inverse problems''~\cite{Vardi_93}. 
Define $p_{\min} := \min_{j \in [m]}\{p_j\}>0$ and consider re-scaling the objective function of \eqref{eq:log_invest} by $1/p_{\min}$, which ensures the coefficient in front of each $\ln(\cdot)$ term is at least $1$.  Then this re-scaled problem is an instance of $\poi$ with $f(u) := -\textstyle\sum_{j=1}^m (p_j/p_{\min}) \ln (u_j)$, $\calK:= \bbR^m_+$, $h := \iota_{\Delta_n}$, $\rvA$ defined by $(\rvA w)_j := r_j^\top w$ for $j \in [m]$, $\theta = 1/p_{\min}$, and $f \in \calB_\theta(\calK)$. 

\vspace{1ex}
\noindent {\em 5.\ Computing the analytic center}~\cite{Nest_04}. 
Given a nonempty solid polytope $\calQ = \{x\in \bbR^n: a_i^\top x\ge d_i,\;i=1,\ldots,m\}$, 
the function
\begin{equation}
b(x) := -\textstyle\sum_{i=1}^m \ln(a_i^\top x- d_i), \quad x\in  \calQ \ ,
\end{equation} is an $m$-self-concordant barrier $\calQ$, see \cite{Nest_94}.  We wish to compute the analytic center of $\calQ$ under $b$, which is the optimal solution to the problem  ${\min}_{x\in\calQ}\, b(x)$.  We can transform this problem into an instance of $(P)$ as follows. Define
\begin{equation}
f(x,t) := -\textstyle\sum_{i=1}^m \ln(a_i^\top x- td_i) = -\sum_{i=1}^m \ln(a_i^\top (x/t)- d_i) - m\ln (t)\ ,
\end{equation}
which is a $m$-logarithmically-homogeneous self-concordant barrier on the conic hull of $\calQ$, denoted by $\cone(\calQ)$ and defined by $$\cone(\calQ) := \cl\{(x,t)\in\bbR^{n+1}: x/t\in\calQ,\; t>0\} = \{x\in\bbR^n: a_i^\top x\ge td_i,\;i=1,\ldots,m,\; t \ge  0\} \ .$$ Then we can formulate the analytic center problem as 
\begin{align*}
{\min}_{(x,t)\in\bbR^{n+1}}\; f(x,t)\quad\st\quad (x,t)\in\cone(\calQ),\;t=1 \ ,
\end{align*}
which is an instance of $(P)$ with $f(u) := -\textstyle\sum_{i=1}^m  \ln (u_i)$, $\calK:= \bbR^m_+$, $h = \iota_{\calX}$ where $\calX= \{(x,t)\in\bbR^{n+1} : x \in \calQ, \ t=1\}$,  $\rvA$ defined by $(\rvA (x,t))_i := a_i^\top x- td_i$ for $i \in [m]$, and $\theta = m$.

The above formulation can be generalized to any nonempty convex compact set $\calQ\subseteq\bbR^{n}$ equipped with a (standard strongly) $\vartheta$-self-concordant barrier $b$ on $\calQ$. 
From~\cite[Proposition~5.1.4]{Nest_94}, there exists a constant $c\le 20$ for which  
\begin{equation}
f(x,t):= c^2(b(x/t)-2\vartheta\ln t) \label{eq:barrier_cone}
\end{equation}
is a $(2c^2\vartheta)$-logarithmically-homogeneous self-concordant barrier on 
$\cone(\calQ)$. 
Therefore using~\eqref{eq:barrier_cone} the analytic center problem can be formulated as an instance of $\poi$ in a similar way as above.



\vspace{1ex}
\noindent {\em 6.\ $D$-optimal design}~\cite{Fedorov_72}.  Given $m$ points $a^1,\ldots,a^m \in \mathbb{R}^n$ whose affine hull is $\mathbb{R}^n$, the $D$-optimal design problem is:
\begin{align}
\min \; h(x):= -\ln\det \big(\textstyle\sum_{i=1}^m x_i a_ia_i^T\big) \quad \st \;\; x\in\Delta_m \ . \label{eq:Dopt2}
\end{align}
In the domain of statistics the $D$-optimal design problem corresponds to maximizing the determinant of the Fisher information matrix $\EE(aa^T)$, see  \cite{kiefer1960equivalence}, \cite{atwood1969optimal}, as well as the exposition in \cite{boyd_04}.  And in computational geometry, $D$-optimal design arises as a Lagrangian dual problem of the minimum volume covering ellipsoid (MVCE) problem, which dates back at least 70 years to \cite{john}, see Todd \cite{toddminimum} for a modern treatment.  Indeed, \eqref{eq:Dopt2} is useful in a variety of different application areas, for example, computational statistics \cite{croux2002location} and data mining \cite{knorr2001robust}.  

A recent extension of \eqref{eq:Dopt2} is the design of a supervised learning pipeline for new datasets to maximize the Fisher information, see the recent paper by Yang et al. \cite{yang2020efficient}.  The optimization problem they consider is the following variant of \eqref{eq:Dopt2}:
\begin{align}
\min &\; h(x):= -\ln\det \big(\textstyle\sum_{i=1}^m x_i a_ia_i^T\big) \label{eq:Dopt2.5}\\ 
\quad \st &\;\; \textstyle\sum_{i=1}^n \bar t_i x_i \le \tau \ \label{eq:Dopt3} \\
\quad &\;\;x_i \in [0,1] \ \mbox{for} \ i \in [m] \ , \label{eq:Dopt4}
\end{align}
where the decision variable $x_i$ models the decision to fit model $i$ or not, $\bar t_i$ is the estimated pipeline running time of pipeline $i$, $\tau$ is the runtime limit, and $a_i$ is the vector of latent meta-features of model $i$ for $i \in [m]$.  The constraints \eqref{eq:Dopt4} are a linear relaxation of the (computationally unattractive) desired combinatorial constraints $x_i \in \{0,1\}$, $i \in [m]$.  We refer the interested reader to \cite{yang2020efficient} for further details and model variations. Here we see that \eqref{eq:Dopt2.5}-\eqref{eq:Dopt4} is an instance of \eqref{poi} with $f(U) := -\ln\det(U)$, $\theta = n$, $\rvA x := \textstyle\sum_{i=1}^m x_i a_ia_i^T$, and $h$ is the indicator function of the feasible region of the constraints \eqref{eq:Dopt3}-\eqref{eq:Dopt4}.

As mentioned earlier, the $D$-optimal design problem was one of the primary motivators for the research in this paper. 
Indeed, Khachiyan proved that his ``barycentric coordinate descent'' method for this problem -- which turns out to be precisely the Frank-Wolfe method (with exact line search) -- has a computational guarantee that is essentially $O\big(n^2/\varepsilon\big)$ iterations to produce an $\varepsilon$-approximate solution.  What has been surprising about this result is that the $D$-optimal design problem violates the basic assumption underlying the premise for the analysis of the Frank-Wolfe method, namely $L$-smoothness.  Khachiyan's proof used original and rather clever arguments that do not easily carry over elsewhere, which has begged the question of whether or not any of Khachiyan's arguments might underlie any general themes beyond $D$-optimal design, and if so what might be the mathematical structures driving any such themes?   We will answer these questions in the affirmative in Section \ref{darkout} by showing that the Frank-Wolfe method achieves essentially $O((\theta + \rh)^2/\varepsilon)$ iteration complexity when used to tackle any problem of the form \eqref{poi}, where $\rh$ is the variation of $h$ on its domain ($\rh := {\max}_{x,y\in\calX}\; \abst{h(x) - h(y)}$) and $f$ is a $\theta$-logarithmically homogeneous self-concordant barrier. When specialized to the $D$-optimal design problem, we have $\rh = 0$ (since $h$ is an indicator function), and $\theta=n$, whereby we recover Khachiyan's $O(n^2/\varepsilon)$ dependence on $\varepsilon$. 
In this respect, our results reveal certain intrinsic connections between $\theta$-logarithmic homogeneity and the Frank-Wolfe method.

Interestingly, we note that historically the theory of self-concordant functions was initially developed to present a general underlying theory for Newton's method for barrier methods in convex optimization.  However, the results in this paper indicate that a subclass of self-concordant functions, namely the class of $\theta$-logarithmically homogeneous self-concordant barriers, are also tied to an underlying theory for the Frank-Wolfe method.

By way of concluding this discussion, we note that the relevant literature contains some first-order methods other than Frank-Wolfe have been proposed to solve problems similar to~\eqref{poi}. For example, \cite{Tran_15} considered~\eqref{poi} with the linear operator $\rvA$ being invertible and the function $f$ being standard self-concordant but not necessarily a logarithmically-homogeneous barrier. The authors in \cite{Tran_15} proposed a proximal gradient method for solving this problem, and showed that the method globally and asymptotically converges to the unique optimal solution. However, the global convergence rate of this method was not shown.

\subsection{Contributions}\label{sec:contribution}

We summarize our contributions as follows:
\begin{enumerate}
\item We propose a generalized Frank-Wolfe method  for solving $\poi$ with $f\in\calB_\theta(\calK)$.  We show that the Frank-Wolfe method requires $O((\delta_0 + \theta + \rh)\ln(\delta_0) + (\theta+\rh)^2/\varepsilon)$ iterations to produce an $\varepsilon$-approximate solution of $\poi$, namely $x\in\dom F$ such that $F(x) - F^*\le \varepsilon$,  where $\delta_0$ denotes the initial optimality gap and $\rh$ denotes the variation of $h$ on its domain. This iteration complexity bound depends on just three (natural) quantities associated with $\poi$: (i) the initial optimality gap $\delta_0$, (ii) the complexity parameter $\theta$ of $f$, and (iii) the variation of $h$ on $\calX$.  When $h$ is the sum of a convex quadratic function and an indicator function, our algorithm specializes to that in Dvurechensky et al.\ \cite{Dvu_20}. However, our iteration complexity bounds are quite different from that in~\cite{Dvu_20} -- in particular our bounds are affine invariant, more naturally interpretable, and are typically easy to estimate and can yield an {\it a priori} bound on the number of iterations needed to guarantee a desired optimality tolerance $\varepsilon$.  These issues are discussed in details in Remark \ref{rmk:Dvu}. 
\item Our analysis also yields $O((\delta_0 + \theta + \rh)\ln(\delta_0) + (\theta+\rh)^2/\varepsilon)$ iteration complexity to produce $x\in\dom F$ whose Frank-Wolfe gap (defined in~\eqref{eq:FW_gap} below) is no larger than $\varepsilon$. Since the  Frank-Wolfe gap is constructed at each iteration and is often used as the stopping criterion for the Frank-Wolfe method, our result provides a further constructive bound on the number of iterations required to detect a desired optimality tolerance. 

\item When specialized to the $D$-optimal design problem, our general algorithm almost exactly recovers the iteration complexity of Khachiyan's specialized method for $D$-optimal design in~\cite{khachiyan1996rounding}. Indeed, the complexities of these two methods have identical dependence on $\varepsilon$, namely $n^2/\varepsilon$.  However, Khachiyan's specialized method has an improved ``fixed'' term over our general method by a factor of $\ln(m/n)$; see Remark~\ref{autumn} for details.

\item We present a mirror descent method with adaptive step-size applied to the (Fenchel) dual problem $\doi$ of $\poi$. The dual problem $\doi$ shares a somewhat similar structure to $\poi$ in that its objective function is non-smooth and non-Lipschitz. However, unlike $(P)$, the objective function has unbounded domain.  Although these features make the direct analysis of mirror descent rather difficult, through the duality of mirror descent and the Frank-Wolfe method we provide a computational complexity bound for this mirror descent method via the Frank-Wolfe method applied to $\poi$.   An application of our mirror descent method for $\doi$ arises in the sub-problem in Bregman proximal-point methods. 
\item We apply our method to the TV-regularized Poisson image de-blurring problem.  We present computational experiments that point to the potential usefulness of our generalized Frank-Wolfe method on this imaging problem in Section \ref{sec:delurring}.
\end{enumerate}


\subsection{Outline and Notation} 
\noindent {\bf Outline.} The paper is organized as follows. In Section~\ref{darkout} we present and analyze our generalized Frank-Wolfe method for $\poi$ when $f \in \calB_\theta(\calK)$, using an adaptive step-size strategy that is a natural extension of the strategy developed in~\cite{Dvu_20}. In Section \ref{doodle} we study the (Fenchel) dual $\doi$ of ${\poi}$ and derive and analyze a dual mirror descent method for solving $\doi$ based on the generalized Frank-Wolfe method for solving $(P)$. 
In Section~\ref{experiments} we present computational experiments that point to the potential usefulness of our generalized Frank-Wolfe method on Poisson image de-blurring problems with TV regularization, and we also present computational experiments on the PET problem.

\vspace{1ex}
\noindent {\bf Notation.} Let $\mathbb{R}^n_+ := \{ x \in \mathbb{R}^n : x \ge 0 \}$ and $\mathbb{R}^n_{++} := \{ x \in \mathbb{R}^n : x > 0 \}$. The set of integers $\{1,\ldots,n\}$ is denoted by $[n]$. The domain of a convex function $f$ is denoted by $\dom f :=\{x \in \mathbb{R}^n : f(x) < \infty \}$.  We use $H(x)$ to denote the Hessian of the function $f$.  The interior and relative interior of a set $\calS$ are denoted by $\inter\calS$ and $\relint \calS$, respectively. We use $e$ to denote the vector with entries all equal to 1, $\diag(x)$ to denote the diagonal matrix whose diagonal entries correspond to those of $x$, and $\Delta_n$ to denote the standard $(n-1)$-dimensional simplex in $\mathbb{R}^n$, namely $\Delta_n := \{ x \in \mathbb{R}^n : \sum_{i=1}^n x_i =1, \ x \ge 0 \}$. We use $\mathbb{S}_{+}^{n}$ ($\mathbb{S}_{++}^{n}$) to denote the set of $n\times n$ symmetric positive semidefinite (positive definite) matrices, and write $B\in\mathbb{S}_{+}^{n}$ as $B\succeq 0$ and $B\in\mathbb{S}_{++}^{n}$ as $B\succ 0$. The $p$-norm of a vector $x\in\bbR^n$ is denoted and defined by $\normt{x}_p = (\sum_{i=1}^n \abst{x_i}^p)^{1/p}$. 

\section{A generalized Frank-Wolfe method for $\poi$ when $f$ is a \\$\theta$-logarithmically-homogeneous self-concordant barrier}\label{darkout}

In this section we present and analyze a generalized Frank-Wolfe method for the composite optimization problem $\poi$ in the case when $f \in \calB_\theta(\calK)$, using an adaptive step-size strategy that is a natural extension of the strategy developed in Dvurechensky et al. \cite{Dvu_20}.  We assume throughout that $\calX := \dom h$ is a convex and  compact set, and that $\rvA(\calX)\cap\dom f\ne \emptyset$.  These two assumptions together with the differentiability of $f$ on $\inter\calK$ ensure that $\poi$ has at least one optimal solution which we denote by $x^*$ and hence $F^*=F(x^*)$.  

Let us introduce some important notation.  For any $u \in \inter\calK$, the Hessian $H(u)$ of $f$ is used to define the local (Hilbert) norm $\| \cdot \|_u$ defined by:
$$ \|w \|_u := \sqrt{\lranglet{w}{H(u)w}} \ \ \mathrm{for~all~} w \in \mathbb{R}^m \ . $$ 
 The Dikin ball $\calD(u,1)$ at $u\in \inter\calK$ is defined by $$ \calD(u,1) := \{v\in\calK:\normt{v-u}_u<1\} \ , $$ and it can be shown  that $\calD(u,1)\subseteq \inter\calK$, see Nesterov and Nemirovski~\cite[Theorem~2.1.1]{Nest_94}.  The self-concordant function $f$ is well-behaved inside the Dikin ball as we will review shortly.   Define the univariate function $\omega$ and its Fenchel conjugate $\omega^*$ as follows:
\begin{equation}\label{eq:omegas} \omega(a) := -a - \ln(1-a) \;\; \forall\,  a < 1 \ , \ \ \   \mathrm{and} \ \  \omega^*(a) := a - \ln(1+a) \;\; \forall\,  a > -1  \ . 
\end{equation}
It turns out that $f$ can be nicely upper-bounded inside the Dikin ball, namely:
\begin{align}
f(v)&\le f(u) + \lranglet{\nabla f(u)}{v-u} + \omega\left(\normt{v-u}_u\right) \ \forall\,u\in\inter\calK, \;\forall \, v \in \calD(u,1) \ ,\label{eq:curvature_SC}
\end{align}
see Nesterov~\cite[Theorem~4.1.8]{Nest_04}. 

We now develop our generalized Frank-Wolfe method for the composite optimization problem $\poi$ under the condition that $f \in \calB_\theta(\calK)$, and whose formal rules are presented in Algorithm~\ref{algo:FW_SC}. First, we choose any starting point $x^0\in\calX$ such that $\rvA x^0\in\dom f(=\inter\calK)$, namely $x^0\in\bar{\calX}:=\calX\cap\rvA^{-1}(\dom f)$. Indeed, from the description below, we will see that the whole sequence of iterates $\{x^k\}_{k\ge 0}$ generated by our method lies in $\bar{\calX}$. Given the current iterate $x^k\in\bar{\calX}$,
 the method first computes the gradient $\nabla f({\rvA}x^k)$ and then solves for a minimizer $v^k$ of the generalized Frank-Wolfe sub-problem given by
$$v^k\in{\argmin}_{x\in\bbR^n} \lranglet{\nabla f(\rvA x^k)}{\rvA x} + h(x)\ .$$  The next iterate is then determined as a convex combination of $x^k$ and $v^k$: $x^{k+1} \gets x^k + \alpha_k (v^k - x^k)$ for some $\alpha_k \in [0,1]$, where $\alpha_k$ is the step-size at iteration $k$.  For $L$-smooth functions $f$, the step-size can be chosen by a variety of strategies, including simple rules such as $\alpha_k = \tfrac{2}{(k+2)}$, exact line-search to minimize $f(x^k + \alpha(v^k - x^k))$ on $\alpha\in [0,1]$, or an adaptive strategy based on curvature information, etc.  Here we present an adaptive strategy based on the upper bound model \eqref{eq:curvature_SC}, which can also be viewed as an extension of the adaptive strategy used in \cite{Dvu_20} and which itself is an extension of the adaptive strategy in Demyanov and Rubinov~\cite{dem1967minimization}.  Define the Frank-Wolfe gap (``FW-gap") $G_k$ by
\begin{equation}
G_k := \lranglet{\nabla f(\rvA x^k)}{\rvA(x^k - v^k)}+ h(x^k) - h(v^k) \ , \label{eq:FW_gap}
\end{equation}
and the optimality gap $\delta_k := F(x^k) - F^*$. Note that $G_k \ge 0$ and in fact by the convexity of $f$ it holds that 
\begin{align}
\delta_k &= (f(\rvA x^k) - f(\rvA x^*)) + (h(x^k) - h(x^*))\nn\\
&\le \lranglet{\nabla f(\rvA x^k)}{\rvA(x^k - x^*)} + (h(x^k) - h(x^*))\nn\\
&\le \lranglet{\nabla f(\rvA x^k)}{\rvA(x^k - v^k)} + (h(x^k) - h(v^k)) = G_k \ , \label{squirrels}
\end{align} 
hence $G_k$ is indeed an upper bound on the optimality gap $\delta_k$.  Denoting $D_k := \|\rvA(v^k-x^k)\|_{\rvA x^k}$, we then have that  for any $\alpha\ge 0$, 
\begin{equation}
f(\rvA x^{k} + \alpha\rvA (v^k-x^k)) \le f(\rvA x^k) - \alpha \lranglet{\nabla f(\rvA x^k)}{\rvA (x^k - v^k)} + \omega\left(\alpha D_k\right)  \ . 
\label{eq:curvature_SC_algo}
\end{equation} 
(Note that if $\alpha<1/D_k$, then~\eqref{eq:curvature_SC_algo} follows from \eqref{eq:curvature_SC}; otherwise, by the definition of $\omega$ in~\eqref{eq:omegas}, we have $\omega\left(\alpha D_k\right)=+\infty$ and~\eqref{eq:curvature_SC_algo} still holds.)   
Also, by the convexity of $h$, we have
\begin{equation}
h(x^{k} + \alpha(v^k-x^k))\le (1-\alpha)h(x^{k}) + \alpha h(v^k) = h(x^{k}) - \alpha(h(x^k) - h(v^{k})) \ .\label{eq:ub_h}
\end{equation}
Adding~\eqref{eq:curvature_SC_algo} and~\eqref{eq:ub_h} together, we obtain 
\begin{equation}
F(x^{k} + \alpha(v^k-x^k))\le F(x^k) - \alpha G_k + \omega\left(\alpha D_k\right) \ , \quad \forall\,\alpha\ge 0  \ , \label{eq:ub_F}
\end{equation}
and optimizing the right-hand-side over $\alpha \in [0,1]$ 
yields the step-size:
$$ \alpha_k := \min\left\{\frac{G_k}{D_k(G_k+D_k)}, 1\right\}.$$ 
Notice that this step-size specializes precisely to the adaptive step-size developed in \cite{Dvu_20} in the case when the function $h$ is the indicator function $\iota_{\calX}$ of a compact convex set $\calX$, since in this case $h(x^k)=h(v^k)=0$ for all $k$ and hence $G_k$ turns out to be the standard ``gap function'' $\lranglet{\nabla f(\rvA x^k)}{\rvA(x^k - v^k)}$ as used in \cite{Dvu_20}. In addition, note that the step-size $\alpha_k$ ensures that $x^{k+1}\in\bar{\calX}$. 

\begin{algorithm}[t!]
\caption{(generalized) Frank-Wolfe Method for composite optimization involving $f\in \calB_\theta(\calK)$ with adaptive step-size}\label{algo:FW_SC}
\begin{algorithmic}
\State {\bf Input}: Starting point $x^0\in\bar{\calX}:=\calX\cap\rvA^{-1}(\dom f)$ 
\State {\bf At iteration $k\in\{0,1,\ldots\}$}:
\begin{enumerate}
{\setlength\itemindent{10pt} \item \label{item:LMO}  Compute $\nabla f(\rvA x^k)$ and $v^k\in\argmin_{x\in\bbR^n} \lranglet{\nabla f(\rvA x^k)}{\rvA x} + h(x)$} 
{\setlength\itemindent{10pt} \item \label{item:step_size_SC} Compute $G_k := \lranglet{\nabla f(\rvA x^k)}{\rvA(x^k - v^k)}+ h(x^k) - h(v^k)$ and $D_k:=  \|\rvA(v^k-x^k)\|_{\rvA x^k}$ , and compute the step-size:
\begin{align}
\alpha_k := \min\left\{\frac{G_k}{D_k(G_k+D_k)} \ , 1\right\} \label{eq:step_size_SC}
\end{align}
\setlength\itemindent{10pt} \item Update $x^{k+1}:= x^k + \alpha_k (v^k-x^k)$ 
}
\end{enumerate}
\end{algorithmic}
\end{algorithm}


Before presenting our analysis of Algorithm \ref{algo:FW_SC} we make two remarks. First, notice that the complexity parameter $\theta$ of $f$ is not needed to run Algorithm~\ref{algo:FW_SC}, but it will play a central role in analyzing the iteration complexity of the algorithm.  In a sense, Algorithm~\ref{algo:FW_SC} automatically adapts to the value of $\theta$. 
Second, for most applications --- including all of the applications discussed in Section~\ref{sec:applications} --- the computational cost of computing $D_k$ (in Step~\ref{item:step_size_SC}) is of the same order as computing $G_k$, and grows linearly in the ambient dimension of the variable $x$.  (Note that here our discussion focuses on the dependence of the computational cost on the dimension of $x$ only.) 
To see this, let us fix any $v,x\in\dom F$. For the first application in Section~\ref{sec:applications} (where $N$ denotes the dimension), both $u:=\rvA x$ and $w:=\rvA v$ can be computed in $O(qN)$ time, due to the fact that the matrix representation of $\rvA$ has only $O(qN)$ nonzeros. 
Since $D_k^2= \sum_{l=1}^N (u_l-w_l)^2/u_l^2$, it  can be computed in $O(qN)$ time. Using similar reasoning, we easily see that for the second, third and fourth applications (where $n$ denotes the dimension), $D_k$ can be computed in $O(mn)$ time. For the last application (namely the $D$-optimal design problem), where the dimension is denoted by $m$, $D_k$ can be computed in $O(mn^2+n^3)$ time for $k=0$ and $O(n^2)$ time for $k\ge 1$; for details see Appendix~\ref{app:D_k}. 



\subsection{Computational guarantees for Algorithm \ref{algo:FW_SC}} \label{sec:comp_guarantee}

We now present our computational guarantees for Algorithm \ref{algo:FW_SC}.  These guarantees depend on only three natural quantities associated with $\poi$: (i) the initial optimality gap $\delta_0$, (ii) the complexity parameter $\theta$ of $f$, and (iii) the {\em variation} of $h$ on $\calX$, which is defined as:
\begin{equation}\label{range}
\rh := {\max}_{x,y\in\calX}\; \abst{h(x) - h(y)} \ . 
\end{equation} 
Regarding the variation $\rh$ we mention two cases in particular: 
\begin{enumerate}
\item when $h = \iota_{\calX}$, we have $\rh = 0$, and 
\item when $h$ is $L$-Lipschitz on $\calX$ with respect to some norm $\normt{\cdot}$, we have $\rh\le L D_{\calX,\normt{\cdot}}$, where 
\begin{equation}
D_{\calX,\normt{\cdot}}:= {\sup}_{x,x'\in\calX}\,\normt{x-x'}<+\infty \label{eq:diameter}
\end{equation}
is the diameter  of $\calX$ under $\normt{\cdot}$.  And in particular if $h=\normt{\cdot}$, then $\rh \le D_{\calX,\normt{\cdot}}$. 
\end{enumerate}


\begin{theorem} \label{thm:SC}
Suppose that $f \in \calB_\theta(\calK)$ and that Algorithm~\ref{algo:FW_SC} is initiated at $x^0 \in \bar{\calX}$. Let $\delta_0$ denote the initial optimality gap.
\begin{enumerate}
\item \label{item:behavior_Algo_SC} At iteration $k$ of Algorithm~\ref{algo:FW_SC} the following hold:
\begin{enumerate}[leftmargin = 2.5em]
\item If $G_k > \theta + \rh$, then the optimality gap decreases at least linearly at the iteration:
\begin{equation}
\delta_{k+1}\le \left(1-\frac{1}{5.3(\delta_0+\theta+\rh)}\right)\delta_k \ , \label{eq:rmflin_conv_SC}
\end{equation}
\item  If $G_k \le \theta + \rh$, then the inverse optimality gap increases by at least a constant at the iteration:
\begin{equation}
\frac{1}{\delta_{k+1}} \ge \frac{1}{\delta_k} + \frac{1}{12 (\theta+\rh)^2} \ , \ \ \ \mbox{and} \label{eq:sublin_conv_SC}
\end{equation}
\item \label{item:K_lin} The number of iterations $K_{\mathrm{Lin}}$ where $G_k > \theta+\rh$ occurs is bounded from above as follows: $K_{\mathrm{Lin}} \le \lceil 5.3(\delta_0 + \theta+\rh)\ln(10.6\delta_0) \rceil$. 
\end{enumerate}

\item \label{item:K_eps} Let $K_\varepsilon$ denote the number of iterations required by Algorithm \ref{algo:FW_SC} to obtain $\delta_k \le \varepsilon$.  Then:
\begin{equation}
K_\varepsilon   \le  \lceil 5.3(\delta_0 + \theta+\rh)\ln(10.6\delta_0) \rceil +  \left\lceil12(\theta+\rh)^2 \max\left\{\frac{1}{\varepsilon} - \frac{1}{\delta_0} \ , 0 \right\}\right\rceil \ .
\label{rob6miles}
\end{equation}

\item \label{item:FWGAP} Let $\mathrm{FWGAP}_\varepsilon$ denote the number of iterations required by Algorithm \ref{algo:FW_SC} to obtain $G_k \le \varepsilon$.  Then:
\begin{equation}
\mathrm{FWGAP}_\varepsilon   \le  \lceil 5.3(\delta_0 + \theta+\rh)\ln(10.6\delta_0) \rceil + \left\lceil \frac{24(\theta+\rh)^2}{\varepsilon} \right\rceil \ . 
\label{rob7miles}
\end{equation}
\end{enumerate}
\end{theorem}\qed

Before we present our proof, let us make a few remarks about the results in Theorem \ref{thm:SC}.

\begin{remark}[Discussion of complexity results]\label{itslate}
Theorem~\ref{thm:SC} indicates that if $f\in\calB_\theta(\calK)$, then the iteration complexity to obtain an $\varepsilon$-optimal solution using Algorithm~\ref{algo:FW_SC} is 
\begin{equation}
O\big((\delta_0 + \theta + \rh)\ln(\delta_0) + (\theta + \rh)^2/\varepsilon\big) \ ,  \label{eq:complexity_SC}
\end{equation}
which only depends on three measures, namely (i) the logarithmic homogeneity constant (also known as the ``complexity value'') $\theta$ of $f$, (ii) the initial optimality gap $\delta_0$, and (iii) the variation $\rh$ of $h$ on its domain, in addition to the desired optimality gap $\varepsilon$.  Furthermore, in most of the applications discussed in Section \ref{sec:applications} (namely applications (2.), (4.), (5.), and (6.)) we have $h = \iota_{\calX}$ for some $\calX$ and hence $\rh = 0$, and so the iteration complexity depends only on $\theta$ and $\delta_0$.

It is interesting to note in the case when $h=\iota_\calX$ is the indicator function of a compact region $\calX$, that the iteration complexity bound \eqref{eq:complexity_SC} does not rely on any measure of size of the feasible region $\calX$, since in this case $\rh = 0$.  And even when $\rh >0$, the complexity bound \eqref{eq:complexity_SC} has no specific dependence on the behavior of $\bar f$ on $\calX$.  In this way we see that the only way that the behavior of $\bar f$ enters into the iteration complexity is simply through the value of $\theta$.  (This is in contrast to the traditional set-up of the Frank-Wolfe method for $L$-smooth optimization, where the fundamental iteration complexity depends on a bound on the {\em curvature} of the function $\bar f$ on the feasible region $\calX$ -- which we will discuss later in Remark \ref{rmk:Dvu} -- which in turn can grow quadratically in the diameter of the feasible region, see for example \cite{Jaggi_13,Freund_16}.)

We also note that the iteration complexity bound \eqref{eq:complexity_SC} is also independent of any properties of the linear operator $\rvA$, and in this way its dependence on a specific data instance $\rvA$ is only through the initial optimality gap $\delta_0$.  Therefore \eqref{eq:complexity_SC} is data-instance independent except for the way that the data $\rvA$ affects the initial optimality gap.
\end{remark}

\begin{remark}[Comparison with Khachiyan~\cite{khachiyan1996rounding} for $D$-optimal design]\label{autumn}
Let us specialize the complexity bound in~\eqref{eq:complexity_SC} to the $D$-optimal design problem in~\eqref{eq:Dopt2}, and compare it with the complexity bound in Khachiyan~\cite{khachiyan1996rounding}.  Note that for the problem in~\eqref{eq:Dopt2} we have $\theta=n$, i.e., the dimension of the ambient space of the points $a_1, \ldots, a_m$.  In addition, if we choose the starting point $p^0 = (1/m)e$, where $e:=(1,\ldots,1)\in\bbR^m$, 
then  
\begin{equation}
\delta_0\le n\ln(m/n)  \label{eq:bound_delta_0}
\end{equation}
(which we show in Appendix~\ref{app:proof_bound_delta_0}), and then based on~\eqref{eq:bound_delta_0} and $\rh =0$, the itration complexity bound in~\eqref{eq:complexity_SC} becomes
\begin{equation}
O\big(n\ln(m/n)(\ln n + \ln\ln(m/n)) + n^2/\varepsilon\big) \ . \label{eq:comp_Dopt_ours}
\end{equation}
Using the same starting point, Khachiyan's Frank-Wolfe method \cite{khachiyan1996rounding} uses exact line-search (based on a clever observation from the Inverse Matrix Update formula \cite{hager}), and attains the complexity bound
\begin{equation}
O\big(n(\ln n + \ln\ln(m/n)) + n^2/\varepsilon\big) \ . \label{eq:comp_Dopt_Khachiyan}
\end{equation}
Observe that~\eqref{eq:comp_Dopt_ours} has the exact same dependence on $\varepsilon$ as~\eqref{eq:comp_Dopt_Khachiyan}, namely $O(n^2/\varepsilon)$, but its first term 
is inferior to~\eqref{eq:comp_Dopt_Khachiyan} by the factor $O(\ln(m/n))$.  The improvement in the first term of Khachiyan's bound over the bound in \eqref{eq:comp_Dopt_ours} is due to his improved estimate of the linear convergence rate in the case $G_k > \theta$ in Algorithm~\ref{algo:FW_SC}, which arises from exploiting an exact line-search on $f$.  This is in contrast with our method, which only does an exact line-search on  the upper bound model of $f$ in \eqref{eq:ub_F}. A detailed analysis of this last point is given in Remark \ref{stormymonday} at the end of this section.
\end{remark}

\begin{remark}[Comparison with Dvurechensky et al.~\cite{Dvu_20}] \label{rmk:Dvu} 
The recent work of Dvurechensky et al.~\cite{Dvu_20} considers the Frank-Wolfe method for the problem $\min_{x\in\calX}\, \bar F(x)$, 
where $\calX$ is nonempty, convex and compact and $\bar F$ is a {non-degenerate} (strongly) $M$-self-concordant function for some $M>0$. This means $\barF$ is convex and three-times differentiable on $\dom \barF$, $\nabla^2 \barF(x)\succ 0$ for all $x\in\dom \barF$ (this is the definition that $\barF$ is non-degenerate), and 
\begin{equation}
\abst{D^3\barF(x)[z,z,z]}\le 2M^{-1/2}(\lranglet{\nabla^2 \barF(x)z}{z})^{3/2} \  \quad \forall\,x\in\dom \barF \  , \;\;\forall\,z\in\bbR^n \ . \label{eq:SC_def_M}
\end{equation}
For convenience of comparison, henceforth let $M=1$. 
When $h$ is the sum of a convex quadratic function and an indicator function, i.e., $h(x) = \tfrac{1}{2}\lranglet{x}{Qx} + \xi^\top x + \iota_\calX(x)$ for some $Q \succeq 0$ and $\xi\in\bbR^n$,  
then our method coincides with that in Dvurechensky et al.\ \cite{Dvu_20} with $\barF(x) = f(\rvA x) + \tfrac{1}{2}\lranglet{x}{Qx} + \xi^\top x$. 
The complexity bound developed in \cite{Dvu_20} for computing an $\varepsilon$-optimal solution is: 
\begin{equation}
O\left(\sqrt{L(x^0)}D_{\calX,\normt{\cdot}_2}\ln\left(\frac{\delta_0}{\sqrt{L(x^0)}D_{\calX,\normt{\cdot}_2}}\right) + \frac{L(x^0)D_{\calX,\normt{\cdot}_2}^2}{\varepsilon}\right) \ , \label{eq:comp_Dvu}
\end{equation}
where  
\begin{align}
\calS(x^0) &:= \big\{x\in\dom \barF\cap\calX\,:\,  \barF(x) \le \barF(x^0) \big\} \ , \mbox{and} \label{eq:S_Dvu}\\
L(x^0) &:= {\max}_{x\in \calS(x^0)} \;\;\lambda_{\max} (\nabla^2 \barF(x) ) <+\infty \ . \label{eq:L_Dvu}
\end{align}
In~\eqref{eq:L_Dvu} $\lambda_{\max} (\nabla^2 \barF(x))$ denotes the largest eigenvalue of $\nabla^2 \barF(x)$, and in \cite{Dvu_20} it is further assumed that $\barF$ is non-degenerate (which necessarily presumes that $\rank (\rvA) = n $)  
in order to ensure the compactness of $\calS(x^0)$, and hence the finiteness of $L(x^0)$. 

It is instructive to compare the two iteration complexity bounds~\eqref{eq:complexity_SC} and~\eqref{eq:comp_Dvu}, and we note the following advantageous features of our bound \eqref{eq:complexity_SC} as follows: 
\begin{itemize}
\item {\em Affine invariance.} Affine invariance is an important structural property of certain algorithms; for instance Newton's method is affine invariant whereas the steepest descent method is not.  It is well known that the Frank-Wolfe method is affine invariant.  Current state-of-the-art complexity analysis of the Frank-Wolfe method for $L$-smooth functions yields an appropriate affine-invariant complexity bound by utilizing the so-called {\em curvature constant} $C_{\barF}$ of Clarkson \cite{clarkson} defined by 
\begin{equation}
C_{\barF}:= 
\max_{x,y\in\calX,\alpha\in[0,1]}\,\frac{2}{\alpha^2}(\barF(x+\alpha(y-x))-\barF(x) - \alpha\lranglet{\nabla \barF(x)}{y-x}) \ ,
\end{equation}
which is a finite affine-invariant quantity \cite{clarkson} when $\bar F$ is $L$-smooth. (This is the same curvature which was alluded to in Remark \ref{itslate}.) Of course $C_{\barF}$ is not typically finite when $\barF$ is self-concordant.  The complexity bound in \eqref{eq:comp_Dvu} depends on measures that are tied to the Euclidean inner product and norm, namely $D_{\calX,\normt{\cdot}_2}$ and $L(x^0)$, and are not affine-invariant, even though the Euclidean norm plays no part in the algorithm. (Under an invertible affine transformation of the variables these measures will change but the performance of the Frank-Wolfe method will not change.)  In contrast, all the quantities in our complexity bounds, namely $\delta_0$, $\theta$ and $\rh$ are affine-invariant, therefore the complexity bounds in \eqref{eq:complexity_SC} are affine-invariant.

\item {\em Interpretability.} Apart from $\varepsilon$, our complexity result only depends on $\delta_0$, $\theta$, and $\rh$, all of which admit natural behavioral interpretations. 
Specifically, $\delta_0$ measures the initial sub-optimality gap, $\theta$ is the ``complexity parameter'' of the barrier $f$ (in the lexicon of Renegar \cite{Renegar_01}), and $\rh$ measures the variation of $h$ over the set $\calX$. 

\item {\em Ease of parameter estimation.} Note that all of the three parameters $\delta_0$, $\theta$, and $\rh$ in our complexity bound~\eqref{eq:complexity_SC} are either easy to know or  easy to appropriately bound. 
Given a logarithmically-homogeneous self-concordant barrier $f$, its complexity value $\theta$ is typcially known {\it a priori} or can be easily determined using~\ref{item:grad_identity} of Lemma \ref{lem:LHSCB}.  
Since  $\delta_0 \le G_0$, a natural upper bound on $\delta_0$ is the initial FW-gap $G_0$,  which is computed in the second step at iteration $k=0$ of Algorithm \ref{algo:FW_SC}.  Regarding $\rh$, note that $\rh = 0 $ when $h$ is the indicator function of a convex set $\calX$, namely $h = \iota_\calX$.  In the case when $h(x) := \xi^\top x +  \iota_\calX(x)$ then $\rh$ can be computed exactly as the difference of two linear optimization optimal values on $\calX$ or can be upper bounded using $$\rh = {\max}_{x,x'\in\calX}\,\abst{\xi^\top (x-x')}\le \normt{\xi}_*D_{\calX,\normt{\cdot}} \ , $$ in the case when the norm $\normt{\cdot}$ can possibly be chosen to yield easily computable values of $D_{\calX,\normt{\cdot}}$.  Apart from this case, there also exist many other cases where the simple nature of $h$ and $\calX$ yield easily computable upper-bounds on $\rh$. 
\end{itemize}
\end{remark}

\subsection{Proof of Theorem \ref{thm:SC}}

We first state some facts about $\theta$-logarithmically homogeneous self-concordant barrier functions. 

\begin{lemma}[{see Nesterov and Nemirovskii~\cite[Corollary~2.3.1, Proposition~2.3.4, and Corollary 2.3.3]{Nest_94}}]\label{lem:LHSCB}
If $f\in \calB_\theta(\calK)$, 
then for any $u\in\inter\calK$, 
 we have 
\begin{enumerate}[start = 4,label = {\rm (P\arabic*)}, leftmargin = 3.9em]
\item  \label{item:def_SCB} $\abs{\lranglet{\nabla f(u)}{w}} \le \sqrt{\theta}\|w\|_u$ $\quad \forall\,w\in\bbR^m$, 
\item \label{item:recession_cone} $\|v\|_u \le -\lranglet{\nabla f(u)}{v}$ $\quad \forall\,v\in\calK$,
\item \label{item:Hessian_grad} $\lranglet{\nabla f(u)}{w} = -\lranglet{H(u)u}{w}$ $\quad \forall\,w\in\bbR^m$,
\item \label{item:grad_identity} $\lranglet{\nabla f(u)}{u} = -\theta$, and
\item \label{item:thetageone} $\theta \ge 1$. \qed
\end{enumerate}
\end{lemma}

We also introduce some properties of the function $\omega^*$ (cf.~\eqref{eq:omegas}) and present an ``old'' property of the logarithm function. 
\begin{prop}\label{lem:omega_conj}
The function $\omega^*$ is strictly increasing on $[0,+\infty)$, and 
\begin{align}
\omega^*(s)&\ge s^2/3 \quad\;\; \ \forall s \in [0,1/2] \ , \ \mbox{and} \label{eq:omega*_quad}\\
\omega^*(s)&\ge s/5.3 \quad\;\;  \forall s \ge 1/2\ .  \label{eq:omega*_linear}
\end{align}
\end{prop}
\begin{proof}
See Appendix~\ref{app:proof_omega_conj}. 
\end{proof}

\begin{prop}\label{lem:Karmarkar}
\begin{equation}
\ln(1+s)\ge s - \frac{s^2}{2(1-\abs{s})} \quad\forall\ s\in(-1,1) \ . 
\end{equation}
\end{prop}
\begin{proof}
See Appendix~\ref{app:proof_Karmarkar}. 
\end{proof}

We have the following inequality concerning values of $D_k$ and $G_k$ computed in Step \ref{item:step_size_SC} of Algorithm \ref{algo:FW_SC}. For convenience, in the following, define 
$$\tilG_k := \lranglet{\nabla f(\rvA x^k)}{\rvA(x^k - v^k)} \quad \mbox{and}\quad  \beta_k := h(x^k) - h(v^k) \ ,$$
 so that $G_k  = \tilG_k + \beta_k$. Also, by the definition of $\rh$, we know that $\abst{\beta_k}\le \rh$.

\begin{prop}\label{keyinequality} For all $k \ge 0$ it holds that \begin{equation}\label{eq:bound_Dk}
D_k \le G_k +\theta+ \rh \ . \end{equation}
\end{prop}

\begin{proof}  We have:
\begin{align}
D_k^2 &= \lranglet{H(\rvA x^k)\rvA v^k}{\rvA v^k} - 2\lranglet{H(\rvA x^k)\rvA x^k}{\rvA v^k} + \lranglet{H(\rvA x^k)\rvA x^k}{\rvA x^k} \ .\label{eq:bound_Dk_0}
\end{align}
By~\ref{item:recession_cone} we see that 
\begin{equation}\label{eq:bound_Dk_1}
\lranglet{H(\rvA x^k)\rvA v^k}{\rvA v^k} \le \lranglet{\nabla f(\rvA x^k)}{\rvA v^k}^2 = (-\tilG_k + \lranglet{\nabla f(\rvA x^k)}{\rvA x^k})^2 = (\tilG_k + \theta)^2 \ , 
\end{equation}
where the last equality above uses~\ref{item:grad_identity}. 
In addition, from~\ref{item:Hessian_grad} and~\ref{item:grad_identity} we have 
\begin{align}
-2\lranglet{H(\rvA x^k)\rvA x^k}{\rvA v^k} + \lranglet{H(\rvA x^k)\rvA x^k}{\rvA x^k} &= 2\lranglet{\nabla f(\rvA x^k)}{\rvA v^k} - \lranglet{\nabla f(\rvA x^k)}{\rvA x^k}\nn\\
&= -2\tilG_k  + \lranglet{\nabla f(\rvA x^k)}{\rvA x^k}\nn\\
&= -2\tilG_k - \theta \ .
\label{eq:bound_Dk_2}
\end{align}
Combining~\eqref{eq:bound_Dk_0},~\eqref{eq:bound_Dk_1} and~\eqref{eq:bound_Dk_2}, we have 
\begin{align*}
D_k^2 &\le (\tilG_k + \theta)^2 -2\tilG_k - \theta\\
& = (G_k -\beta_k + \theta)^2 + 2\beta_k - (2G_k + \theta)\\
&\le (G_k + \theta)^2 + \beta_k^2 - 2\beta_k (G_k + \theta-1)\nt\label{eq:Dk_1} \\
&\le (G_k + \theta)^2 + \rh^2 + 2\rh (G_k + \theta)\nt\label{eq:Dk_2}\\
&= (G_k + \theta + \rh)^2 \ , \nt\label{eq:Dk_3}
\end{align*}
where in~\eqref{eq:Dk_1} we use $2G_k + \theta\ge 0$ and in~\eqref{eq:Dk_2} we use $\abst{\beta_k}\le \rh$ and  $G_k + \theta\ge \theta\ge 1$. 
\end{proof}

The basic iteration improvement inequality for the Frank-Wolfe was presented in \eqref{eq:curvature_SC_algo}, and the step-size in Algorithm \ref{algo:FW_SC} is given by \eqref{eq:step_size_SC}.  
 In the case when $\alpha_k < 1$, it follows from substituting $\alpha_k = \tfrac{G_k}{D_k(G_k+D_k)}$ from \eqref{eq:step_size_SC} into \eqref{eq:ub_F} that the iteration improvement bound is 
\begin{equation}
F(x^{k+1}) \le F(x^k) -\omega^*\left(\frac{G_k}{D_k}\right). 
\end{equation}
Using the notation $\Delta_k := F(x^k) - F(x^{k+1}) = \delta_k-\delta_{k+1}$, we can write this improvement as:
\begin{align}
\Delta_k =  \delta_{k} - \delta_{k+1} = F(x^k) - F(x^{k+1}) &\ge  \omega^*\left(\frac{G_k}{D_k}\right)\ge 0 \ \mbox{when} \ \alpha_k \ < 1 \ . \label{eq:improv_SC}
\end{align}
Let us now prove \eqref{eq:rmflin_conv_SC} and part~\ref{item:K_lin} of Theorem \ref{thm:SC}.  Since $G_k > \theta + \rh$, by~\ref{item:def_SCB} in Lemma~\ref{lem:LHSCB} we have 
\begin{align*}
D_k  = \normt{\rvA v^k - \rvA x^k}_{\rvA x^k} & \ge  \abst{\lranglet{\nabla f(\rvA x^k)}{\rvA v^k-\rvA x^k}}/\sqrt{\theta} \\
&\ge   \tilG_k/\sqrt{\theta}   = (G_k-\beta_k)/\sqrt{\theta}\ge (G_k-\rh)/\sqrt{\theta} > \sqrt{\theta}\ge 1 \ . \nt \label{eq:D_k_ge_G_k}
\end{align*}
As a result,  $\alpha_k = \tfrac{G_k}{D_k(G_k+D_k)}<1$.  
Consequently, by~\eqref{eq:improv_SC} we have
\begin{equation}
F(x^{k+1})\le F(x^k) - \omega^*\left(\frac{G_k}{D_k}\right) \le F(x^k) - \omega^*\left(\frac{G_k}{G_k + \theta +\rh} \right) \ , \label{eq:1st_phase_descent}
\end{equation} 
where the last inequality uses \eqref{eq:bound_Dk} and the monotonicity of $\omega^*$.  Now notice from the condition $G_k > \theta + \rh$ that ${G_k}/(G_k+\theta+ \rh)>1/2$, whereby invoking \eqref{eq:omega*_linear} we have
\begin{align}
\Delta_k = \delta_k - \delta_{k+1} = F(x^{k}) - F(x^{k+1})&\ge  \omega^*\left(\frac{G_k}{G_k+\theta+\rh}\right) \ge \frac{G_k}{5.3(G_k+\theta+\rh)} \ . \label{eq:improv_SC_lin}
\end{align} 
In addition we have
\begin{equation}
\frac{G_k}{5.3(G_k+\theta+ \rh)}\ge \frac{\delta_k}{5.3(\delta_k+\theta+ \rh)}\ge \frac{\delta_k}{5.3(\delta_0+\theta+ \rh)} \ , \label{eq:improv_SC2}
\end{equation} 
where the first inequality uses the strict monotonicity of the function $c \mapsto c/(c+\theta+ \rh)$ on $[0,+\infty)$, and the second inequality uses the monotonicity of the sequence $\{\delta_k\}_{k\ge 0}$ (see~\eqref{eq:improv_SC}). 
Combining~\eqref{eq:improv_SC_lin} and~\eqref{eq:improv_SC2}, we obtain 
\begin{equation}
\delta_{k+1}\le \left(1-\frac{1}{5.3(\delta_0+\theta+ \rh)}\right)\delta_k \ , \label{eq:lin_conv_SC}
\end{equation}
which proves \eqref{eq:rmflin_conv_SC}.  Furthermore, $G_k>\theta+ \rh$ implies that 
\begin{equation}
\delta_k\ge \Delta_k \ge \frac{G_k}{5.3(G_k+\theta+ \rh)} > \frac{G_k}{5.3(G_k+G_k)} = \frac{1}{10.6} \ .  \label{holymoly}
\end{equation}
Now let $K_{\mathrm{Lin}}$ denote the number of iterations of Algorithm \ref{algo:FW_SC} where $G_k > \theta+\rh$ occurs.  By~\eqref{eq:lin_conv_SC} and \eqref{holymoly} it follows that 
\begin{align}
\frac{1}{10.6} < \delta_0  \left(1-\frac{1}{5.3(\delta_0+\theta+\rh)}\right)^{K_{\mathrm{Lin}}-1} \ , 
\end{align}
which then implies that $K_{\mathrm{Lin}} \le \lceil 5.3(\delta_0 + \theta+\rh)\ln(10.6\delta_0) \rceil$ and thus proving part~\ref{item:K_lin} of Theorem \ref{thm:SC}.

Let us now prove \eqref{eq:sublin_conv_SC} of Theorem \ref{thm:SC}.  Towards doing so, we will establish:
\begin{equation}\label{acorns}
G_k \le \theta +\rh \ \implies \   \Delta_k \ge \frac{G_k^2}{12(\theta+\rh)^2}  \ . 
\end{equation} 
We first consider the case where $\alpha_k=1$, whereby $D_k(G_k+D_k) \le G_k$, which implies that $D_k < 1$, and also can be rearranged to yield 
\begin{equation} 
G_k \ge  \frac{D_k^2}{1-D_k} \ .\label{eq:G_k_D_k}
\end{equation}
In addition, by~\eqref{eq:ub_F} we obtain
\begin{align*}
\Delta_k = F(x^k)-F(x^{k+1}) \ge G_k - \omega(D_k)= G_k +D_k + \ln(1-D_k) \ .\nt \label{eq:Delta_lb_alpha_1}
\end{align*}
By~\eqref{eq:Delta_lb_alpha_1}, Proposition~\ref{lem:Karmarkar}, and~\eqref{eq:G_k_D_k}, we have
\begin{equation}
\Delta_k \ge G_k - \frac{D_k^2}{2(1-D_k)} \ge \frac{G_k}{2} \ ,  \label{eq:Delta_k_lb_lin}
\end{equation} which then implies 
that
\begin{equation}
\Delta_k \ge \frac{G_k}{2} \ge \frac{G_k^2}{2(\theta+\rh)} \ge \frac{G_k^2}{2(\theta+\rh)^2} \ge \frac{G_k^2}{12(\theta+\rh)^2} \ , 
\end{equation}
where the second inequality used $G_k \le \theta+\rh$ and the third inequality used $\theta+\rh\ge 1$.  This establishes \eqref{acorns} for the case when $\alpha_k = 1$.

We next consider the case where $\alpha_k <1$, whereby $\alpha_k = \frac{G_k}{D_k(G_k+D_k)}$, and then by \eqref{eq:improv_SC} we have 
\begin{align*}
\Delta_k &= F(x^k) - F(x^{k+1}) \\
&\ge \omega^*\left(\frac{G_k}{D_k}\right) \ge \omega^*\left(\frac{G_k}{G_k + \theta+\rh}\right) \ge \frac{G_k^2}{3(G_k+\theta+\rh)^2}\ge \frac{G_k^2}{12(\theta+\rh)^2} \ , \nt\label{eq:Delta_k_lb_quasi-quad}
\end{align*}
where the second inequality uses \eqref{eq:bound_Dk} and the monotonicity of $\omega^*$, the third inequality 
uses \eqref{eq:omega*_quad} in conjunction with $G_k/(G_k + \theta+\rh) \le 1/2$, and the fourth inequality uses $G_k \le \theta+\rh$. This establishes \eqref{acorns} for the case when $\alpha_k < 1$, completing the proof of \eqref{acorns}.  It thus follows for $G_k \le \theta+\rh$ that
$$ \delta_k - \delta_{k+1} = \Delta_k \ge \frac{G_k^2}{12(\theta+\rh)^2} \ge \frac{\delta_k \delta_{k+1}}{12(\theta+\rh)^2} \ , $$
where the last inequality follows from $\delta_{k+1} \le  \delta_k \le G_k$, and dividing both sides by $ \delta_k \delta_{k+1}$ and rearranging yields the inequality \eqref{eq:sublin_conv_SC}.  

We next prove \eqref{rob6miles} and \eqref{rob7miles}.  If $\delta_0 \le \varepsilon$ the result follows trivially; thus we assume that $\delta_0 > \varepsilon$. Let $\bar K$ denote the expression on the right-side of \eqref{rob6miles}, and suppose Algorithm \ref{algo:FW_SC} has been run for $\bar K$ iterations. Let $N:=  \lceil12(\theta+\rh)^2\left({1}/{\varepsilon} - {1}/{\delta_0} \right)\rceil$, whereby it follows from part~\ref{item:K_lin} 
of Theorem~\ref{thm:SC} that among the first $\bar K$ iterations, the number of iterations  where $G_k \le \theta+\rh$ is at least $N$.  Thus from \eqref{eq:sublin_conv_SC} it follows that 
$$\frac{1}{\delta_{K_\varepsilon}} \ge \frac{1}{\delta_0} + \frac{N}{12 (\theta+\rh)^2} \ge \frac{1}{\delta_0} + \left(\frac{1}{\varepsilon} - \frac{1}{\delta_0} \right) = \frac{1}{\varepsilon} \ , $$ and rearranging yields part~\ref{item:K_eps} of Theorem~\ref{thm:SC}.

Let $k_0 < k_1 < k_2 < \cdots$ denote indices where $G_k \le \theta + \rh$. 
From \eqref{acorns}, \eqref{squirrels}, and the monotonicity of the sequence $\{\delta_k\}_{k\ge 0}$, it follows for all $j\ge 0$ that 
$$ \delta_{k_{j+1}} \le \delta_{k_j+1} \le \delta_{k_{j}} - \frac{G_{k_{j}}^2}{12 (\theta+\rh)^2} \ \ \ \mbox{and} \ \ \ G_{k_{j}} \ge \delta_{k_{j}} \ . $$
Let $d_j := \delta_{k_{j}}$ and $g_j := G_{k_{j}}$ for all $j\ge 0$, then the nonnegative sequences $\{ d_j \}_{j\ge 0}$ and $\{ g_j \}_{j\ge 0}$ satisfy for all $j\ge 0$: 
$$ d_{{j+1}} \le d_{{j}} - \frac{g_{{j}}^2}{12 (\theta+\rh)^2} \ \ \ \mbox{and} \ \ \ g_{{j}} \ge d_{{j}} \ . $$
Thus $\{ d_j \}_{j\ge 0}$ and $\{ g_j \}_{j\ge 0}$ satisfy the hypotheses of the following elementary sequence proposition using $M=12(\theta+\rh)^2$. (This proposition is 
a slight extension of the standard sequence property for Frank-Wolfe type sequences, and we provide a proof in Appendix \ref{holdenwood}.)

\begin{prop}\label{greatrun} Suppose the two nonnegative sequences $\{ d_j \}_{j\ge 0}$ and $\{ g_j \}_{j\ge 0}$ satisfy for all $j \ge 0$:
\begin{itemize}
\item $d_{j+1} \le d_j - g_j^2/M$ for some $M>0$, and
\item $g_j \ge d_j$.
\end{itemize}
Then for all $j \ge 0$ the following holds:
\begin{equation}\label{eq:rate_d_j} d_{j} \le \frac{M}{j + \frac{M}{d_0}} < \frac{M}{j} \ , \end{equation}
and \begin{equation}\label{eq:rate_g_j} \min\{g_0, \ldots, g_j\}  < \frac{2M}{j} \ . \end{equation} \qed

\end{prop}

\noindent 
Let $\mathrm{FWGAP}_\varepsilon$ be as given in part~\ref{item:FWGAP} of Theorem~\ref{thm:SC}, and let $\tilde K$ denote the expression on the right-side of \eqref{rob7miles}.  Suppose Algorithm \ref{algo:FW_SC} has been run for $\tilde K$ iterations. Let $\tilde N:=  \lceil {24(\theta+\rh)^2}/{\varepsilon} \rceil$, whereby it follows from part~\ref{item:K_lin} of Theorem~\ref{thm:SC} that among the first $\tilde K$ iterations, 
the number of iterations where $G_k \le \theta+\rh$ is at least $\tilde N$. Then, it follows that
$$ \min\{G_0, \ldots, G_{\tilde K}\} \le \min\{G_{k_{0}}, \ldots, G_{k_{\tilde N}}\} = \min\{g_0, \ldots, g_{\tilde N}\} < \frac{2M}{\tilde N} = \frac{24(\theta+\rh)^2}{\tilde N} \le \varepsilon \ , $$ where the strict inequality uses Proposition \ref{greatrun}.  This shows \eqref{rob7miles} and completes the proof of Theorem~\ref{thm:SC}. \qed

{\begin{remark}[Continued discussion from Remark \ref{autumn} comparing Theorem \ref{thm:SC} with Khachiyan~\cite{khachiyan1996rounding} for $D$-optimal design]\label{stormymonday} 
Here $h = \iota_{\Delta_n}$, whereby $\rh = 0$, and the rate of linear convergence for iterates where $G_k > \theta$ in \eqref{eq:rmflin_conv_SC} is order $O(1-1/(n+\delta_0))$ as compared to the rate of $O(1-1/n)$ proved in~\cite{khachiyan1996rounding} specifically for the $D$-optimal design problem  with exact line-search. 
Due to the very special structure of the $D$-optimal design problem, the exact line-search is in closed-form, and it enables Khachiyan~\cite{khachiyan1996rounding} to show that the optimality gap improvement bound \eqref{eq:improv_SC} is instead
\begin{equation}
\delta_{k+1} \le  \delta_{k}  -  \omega\left(\frac{G_k}{G_k+\theta}\right) \ . \label{eq:improved_lin_rate}
\end{equation}
Notice that $\omega$ is larger than $\omega^*$, and all the moreso for larger values of its argument, which corresponds to $G_k > \theta$; this then leads to an improved guaranteed linear rate of Khachiyan's algorithm in the case when $G_k > \theta$. However, we stress that the stronger estimate in~\eqref{eq:improved_lin_rate} is rather specific to the $D$-optimal design problem, and we do not expect it to hold in general for $f\in\calB_\theta(\calK)$. 
\end{remark}
}

\section{A Mirror Descent Method for the Dual Problem}\label{doodle}


In this section we present a mirror descent method with adaptive step-size applied to the (Fenchel) dual problem of $\poi$. 
We denote the dual problem of $\poi$ by $\doi$, which is given by:
\begin{equation}
\doi:\ \ \ -d^*:= -{\min}_{y\in\bbR^m} \;[d(y):= f^*(y) + h^*(-\rvA^*y)] \ , \label{doi}
\end{equation}
where $f^*$ and $h^*$ are the Fenchel conjugates of $f$ and $h$, respectively, and $\rvA^*:\bbR^{m}\to \bbR^{n}$ denotes the adjoint of $\rvA$.  We observe the following properties related to $\doi$: 
\begin{enumerate}
\item $f^*$ is a $\theta$-logarithmically-homogeneous self-concordant barrier on the polar of ${\calK}$, namely ${\calK}^\circ:= \{y\in\bbR^m:\lranglet{y}{u}\le 0\;\forall\,u\in{\calK}\}$, and $\calK^\circ$ is also a regular cone.  This follows from~\cite[Theorem~2.4.4]{Nest_94}.
\item $h^*$ is Lipschitz (but not necessarily differentiable) on $\bbR^n$. Indeed, let $\normt{\cdot}$ be a given norm on the space of $x$ variables, 
and define $R_\calX:= \max_{x \in \calX} \|x\|$. 
Since $\calX=\dom h$ is compact and $h$ is closed, 
it follows that $R_\calX<+\infty$  and $h^*$ is $R_\calX$-Lipschitz on $\bbR^n$. 
\item $F^*= -d^*$ and $\doi$ has at least one optimal solution.  Indeed, since $\rvA(\calX)\cap\dom f\ne \emptyset$ and $f$ is continuous on $\dom f$, the strong duality and attainment follows from~\cite[Theorem~3.51]{Peyp_15}.
\end{enumerate}
Although $\doi$ has a similar structure as $\poi$, certain key differences are unattractive in regards to the application of first-order methods for solving $\doi$.  One key difference is that the domain of the dual function $d$ is unbounded, which is in obvious contrast to the bounded domain of the primal function $F$.  This makes it difficult or prohibitive to apply a Frank-Wolfe type method to solve $\doi$. Furthermore, and similar to $\poi$, $\nabla f^*$ does not satisfy either uniform boundedness or uniform Lipschitz continuity on ${\calK}^\circ$, thereby preventing the application of most other types of first-order methods.  Nevertheless, below we present a mirror descent method with adaptive step-size for $\doi$.  Although the lack of good properties prevents the direct analysis of mirror descent in the usual manner, through the duality of mirror descent and the Frank-Wolfe method we provide a computational complexity bound for this mirror descent method. 

Before presenting our mirror descent method for tackling $\doi$, we first present an important application of $\doi$ in the Bregman proximal-point method.

{\em Application in the Bregman proximal-point method (BPPM)~\cite{Censor_92,Ecks_93,Aus_99}.}\; Consider the convex non-smooth optimization problem $\min_{y\in\bbR_+^m}\, \xi(y)$, where $\xi:\bbR^m\to\bbR$ is assumed to be Lipschitz on $\bbR^m$. At the $k$-th iteration of BPPM, one solves  the following problem:
\begin{equation}
y^{k+1}:= {\argmin}_{y\in\bbR^m}\; \xi(y) + \beta_k^{-1} D_\zeta(y,y^k) \ , \label{eq:BPPM} 
\end{equation}
where $y^k$ is the $k$-th iterate of BPPM, $\beta_k>0$ is the step-size, and $\zeta:\bbR_{++}^m\to\bbR$ is the prox function that induces the Bregman divergence 
 \begin{equation}
 D_{\zeta}(y,y^k) := \zeta(y) - \zeta(y^k) - \lranglet{\nabla \zeta(y^k)}{y- y^k} \ .\label{eq:Bregman}
 \end{equation}
 As pointed out in~\cite{Aus_99}, one of the  standard choices of $\zeta$ is $\zeta(y)= -\sum_{i=1}^m  \ln (y_i)$,  and under this choice, if $y^0\in \bbR_{++}^m = \inter \bbR_{+}^m$, then $y^k\in \bbR_{++}^m $ for all $k\ge 1$, and so the constraint set $\bbR_+^m$ is 
automatically taken care of by the prox-function $\zeta$. 
From~\eqref{eq:BPPM} and~\eqref{eq:Bregman} we note that \eqref{eq:BPPM} is in the form of $\doi$ with $f^*(y) := \zeta(y)= -\sum_{i=1}^m \ln (y_i)$, $\rvA=-\rvI$ and $h^*(-\rvA^*y) =\beta_k \xi(y) -\lranglet{\nabla\zeta(y^{k})}{y}$.
 
Our mirror descent method for $\doi$ is shown in Algorithm~\ref{algo:DMD}, and is based on using the function $f^*$ itself as the prox function to induce the Bregman divergence:
$$D_{f^*}(y,y^k) := f^*(y) - f^*(y^k) - \lranglet{\nabla f^*(y^k)}{y- y^k} \ , $$ 
(When $f$ is $L$-smooth, similar ideas have appeared in some previous works, for example Grigas \cite{grigas2016}, Bach \cite{Bach_15}, as well as Lu and Freund~\cite{LuFreund}.) 
In step~\ref{item:subgrad}, we compute a subgradient of the dual function $d$ at $y^k$, which is denoted by $g^k$. In step~\ref{item:u_DMD} we update the primal variables $z^k$ which are used in the method to adaptively determine the step-size in the next step.  In step~\ref{item:step_size_DMD}, we compute the step-size $\gamma_k$, which we will show to be same as the step-size $\alpha_k$ in the Frank-Wolfe method (shown in Algorithm~\ref{algo:FW_SC}). Equipped with $y^k$, $g^k$ and $\gamma_k$, in step~\ref{item:Bregman} we perform a Bregman proximal minimization step to obtain  $y^{k+1}$. We emphasize that  different from the classical mirror descent method (e.g.,~\cite{Nemi_79}), in step~\ref{item:Bregman} we use $f^*$ (which is part of the objective function) as the prox function to induce the Bregman divergence $D_{f^*}(\cdot,\cdot)$.  Also notice that the domain of the sub-problem \eqref{eq:Breg_min} is $\inter\calK^\circ$ and it is perhaps not so obvious without further analysis that \eqref{eq:Breg_min} has an optimal solution.

At first glance it appears that Algorithm~\ref{algo:DMD} might not be efficient to implement, since it involves working with a system of linear equations to determine $z^0$ in the Input, and also involves solving the minimization sub-problem in step~\ref{item:Bregman}. However, as we show below, Algorithm~\ref{algo:DMD} corresponds exactly to the generalized Frank-Wolfe method in Algorithm~\ref{algo:FW_SC} for solving $\poi$, which does not involve these computationally expensive steps. This of course implies that Algorithm~\ref{algo:DMD} can be implemented via  Algorithm~\ref{algo:FW_SC} to obtain the primal iterate sequence $\{x^k\}_{k\ge 0}$, and then the dual iterate sequence $\{y^k\}_{k\ge 0}$ is determined by the simple rule $y^k = \nabla f(\rvA x^k)$ for $k \ge 0$.

\begin{algorithm}[t!]
\caption{Mirror descent method for solving $\doi$ using $f^*$ as the prox function} 
\label{algo:DMD}
\begin{algorithmic}
\State {\bf Input}: Starting points $y^0\in \inter \calK^\circ$ and $z^0\in\{z\in\calX:\rvA z = \nabla f^*(y^0)\}$ 
\State {\bf At iteration $k\in\{0,1,\ldots\}$}:
\begin{enumerate}
{\setlength\itemindent{10pt} \item \label{item:subgrad}  
Let $s^k\in\partial h^*(-\rvA^*y^k)$ 
and define 
\begin{equation}
g^k:= \nabla f^*(y^k) -  \rvA s^k \in \partial d(y^k) \ . \label{eq:g_k}
\end{equation}
}
{\setlength\itemindent{10pt} \item \label{item:u_DMD} 
If $k\ge 1$, compute 
\begin{equation}
z^k:= (1-\gamma_{k-1})z^{k-1} + \gamma_{k-1}s^{k-1} \ . \label{eq:z_k}
\end{equation}
}
{\setlength\itemindent{10pt} \item \label{item:step_size_DMD} Compute $\barG_k := \lranglet{g^k}{y^k}+ h(z^k) - h(s^k)$ and $\barD_k:=  \|g^k\|_{\nabla f^*(y^k)}$, and compute the step-size:
\begin{align}
\gamma_k := \min\left\{\frac{\barG_k}{\barD_k(\barG_k+\barD_k)} \ , 1\right\} \ . \label{eq:step_size_DMD}
\end{align}
\setlength\itemindent{10pt} \item \label{item:Bregman} Update 
\begin{equation}
y^{k+1}:= {\argmin}_{y\in\bbR^m}\, \lranglet{g^k}{y} + \gamma_k^{-1} D_{f^*}(y,y^k) \ . \label{eq:Breg_min}
\end{equation}

}
\end{enumerate}
\end{algorithmic}
\end{algorithm}

\begin{theorem}\label{thm:equiv}
Algorithms~\ref{algo:FW_SC} and~\ref{algo:DMD} are equivalent in the following sense: 
If the starting points $x^0$ in Algorithm~\ref{algo:FW_SC} and $z^0$ in Algorithm~\ref{algo:DMD} satisfy $x^0 = z^0$ and $y^0 =\nabla f(\rvA x^0)$, then an iterate sequence of either algorithm exactly corresponds to an iterate sequences of the other.\end{theorem}

Before we prove this theorem, we first recall some properties of conjugate functions.  Let $w:\bbR^p\to\bbR\cup\{+\infty\}$ be a closed convex function and let $w^*$ denote its conjugate function, which is defined by $w^*(g) := \max_u\{\lranglet{g}{u}-w(u)\}$.   Then $w^*:\bbR^p\to\bbR\cup\{+\infty\}$ is a closed convex function, and
\begin{equation}
g \in \partial w(u)  \Longleftrightarrow \ u \in \partial w^*(g) \Longleftrightarrow \lranglet{g}{u} = w(u) + w^*(g) \ . \label{eq:Fenchel_identity}
\end{equation}

\noindent \emph{Proof of Theorem \ref{thm:equiv}}.  Let $\{y^k\}_{k\ge 0}$ be the sequence of iterates of Algorithm \ref{algo:DMD}, and let us also collect the sequences $\{z^k\}_{k\ge 0}$, $\{s^k\}_{k\ge 0}$, $\{g^k\}_{k\ge 0}$, $\{\barG_k\}_{k\ge 0}$, $\{\barD_k\}_{k\ge 0}$, and $\{\gamma^k\}_{k\ge 0}$ generated in Algorithm \ref{algo:DMD}, and use these sequences to define the following five sequences by the simple assignments $x^k := z^k$, $v^k := s^k$, $\alpha^k := \gamma^k$, $G_k := \barG_k$, $D_k := \barD_k$, for $k \ge 0$.  We now show that these five sequences correspond to an iterate sequence of Algorithm \ref{algo:FW_SC}.  Our argument will rely on the following identity:
\begin{equation}
 y^k = \nabla f(\rvA z^k) \ \ \forall k \ge 0 \ , \label{beard}
\end{equation}
which we will prove by induction.  Notice that \eqref{beard} is true for $k=0$ by supposition in the statement of the theorem.  Now let us assume that \eqref{beard} holds for a given iteration $k$, and let us examine the properties of our sequences.  We have
\begin{align*}
v^k :=s^k\in\partial h^*(-\rvA^*y^k) &\;\;\Longrightarrow \;\; -\rvA^*y^k \in \partial h(v^k)\nt\label{eq:min_subdiff}\\
 &\;\;\Longrightarrow \;\; v^k  \in {\argmin}_{x\in\bbR^n}\; \lranglet{\nabla f(\rvA x^k)}{\rvA x} + h(x) \ , \nt\label{eq:min_subdiff2}
\end{align*} 
where \eqref{eq:min_subdiff} follows from the conjugacy properties \eqref{eq:Fenchel_identity}. This shows that $v^k$ satisfies Step~\ref{item:LMO} of iteration $k$ in Algorithm \ref{algo:FW_SC}.  We also have
\begin{align*}
G_k := \barG_k :=&\lranglet{g^k}{y^k}+ h(z^k) - h(s^k) \nt\label{eq:ugh1}\\
 =& \lranglet{\nabla f^*(y^k) - \rvA s^k}{\nabla f(\rvA z^k) } + h(z^k) - h(s^k) \ , \nt\label{eq:ugh2}\\
 =& \lranglet{\rvA x^k - \rvA v^k}{\nabla f(\rvA x^k) } + h(x^k) - h(v^k) \  \nt\label{eq:ugh3}
\end{align*} 
satisfies the definition of $G_k$ in Step~\ref{item:step_size_SC} of iteration $k$ of Algorithm \ref{algo:FW_SC}.  Similarly, we have 
\begin{align*}
D_k := \barD_k := \normt{g^k}_{\nabla f^*(y^k)} = \normt{\nabla f^*(y^k) - \rvA s^k}_{\rvA z^k } = \normt{\rvA z^k - \rvA s^k}_{\rvA z^k }  = \normt{\rvA x^k - \rvA v^k}_{\rvA x^k } \nt\label{eq:ugh4}
\end{align*} 
satisfies the the definition of $D_k$ in Step~\ref{item:step_size_SC} of iteration $k$ in Algorithm \ref{algo:FW_SC}, which then implies similarly for the formula for $\alpha_k$ in \eqref{eq:step_size_SC} of Algorithm \ref{algo:FW_SC}.  Last of all, we prove the inductive step of the equality \eqref{beard}. From the optimality conditions of the optimization problem in \eqref{eq:Breg_min} we have
\begin{align*}
\nabla f^*(y^{k+1}) = \nabla f^*(y^k) - \gamma_k g_k  \;\;\Longrightarrow \;\; \nabla f^*(y^{k+1}) = \nabla f^*(y^{k}) - \gamma_k g_k = (1-\gamma_k) \rvA z^k + \gamma_k \rvA s^k,
\end{align*}
where in the last step we use $\nabla f^*(y^{k}) = \rvA z^k$ from \eqref{beard} and~\eqref{eq:g_k}. Since $z^{k+1}:= (1-\gamma_{k})z^{k} + \gamma_{k}s^{k}$, we have $\nabla f^*(y^{k+1}) = \rvA z^{k+1}$ implies $y^{k+1} = \nabla f(\rvA z^{k+1})$ by conjugacy and completes the proof of \eqref{beard}. This then shows that an iterate sequence of Algorithm \ref{algo:DMD} corresponds to an iterate sequence of Algorithm \ref{algo:FW_SC}.  The reverse implication can also be proved using identical logic as above. 
\qed

We now leverage the equivalence of Algorithms~\ref{algo:FW_SC} and~\ref{algo:DMD} to analyze the iteration complexity of Algorithm \ref{algo:DMD}.  The following proposition relating the duality gap to the Frank-Wolfe gap will be useful.

\begin{prop}\label{bbike} $G_k = d(y^k) + F(x^k)$ for all $k \ge 0$. 
\end{prop}

\begin{proof} We have for all $k \ge 0$: 
\begin{align}
G_k &= \lranglet{\nabla f(\rvA x^k)}{\rvA x^k} - \lranglet{\nabla f(\rvA x^k)}{\rvA v^k} + h(x^k)  - h(v^k)\\
&= f(\rvA x^k)+ h(x^k) + f^*(y^k)+ \lranglet{-\rvA^*y^k}{ v^k}  - h(v^k)\label{eq:dualgap_1}\\
&= f(\rvA x^k)+ h(x^k) + f^*(y^k)+h^*(-\rvA^*y^k) = d(y^k) + F(x^k) \ . \label{eq:dualgap_2}
\end{align} where in~\eqref{eq:dualgap_1} we use  the conjugacy property in~\eqref{eq:Fenchel_identity} and $y^k =\nabla f(\rvA x^k)$, and in~\eqref{eq:dualgap_2} we use $-\rvA^*y^k\in\partial h(v^k)$ (by~\eqref{eq:min_subdiff} and $s^k = v^k$) and~\eqref{eq:Fenchel_identity}.\end{proof}

Since $G_k$ upper bounds the dual optimality gap $d_k:= d(y^k) - d(y^{*})$, from part~\ref{item:FWGAP} of Theorem~\ref{thm:SC} we immediately have the following corollary. 

\begin{corollary}
Let $\mathrm{DGAP}_\varepsilon$ denote the number of iterations required by Algorithm \ref{algo:DMD} to obtain $d_k \le \varepsilon$.  Then:
\begin{equation}
\mathrm{DGAP}_\varepsilon   \le  \lceil 5.3(\delta_0 + \theta+\rh)\ln(10.6\delta_0) \rceil + \left\lceil \frac{24(\theta+\rh)^2}{\varepsilon} \right\rceil \ . 
\end{equation}\qed
\end{corollary}

We end this section with some remarks. First, if one considers $\doi$ directly then its ``primitive'' objects are $f^*$ and $h^*$, and  implementing Algorithm~\ref{algo:DMD} via Algorithm~\ref{algo:FW_SC} requires knowing $h = (h^*)^*$ and also the Hessian of $f = (f^*)^*$ (used to compute the step-size $\gamma_k$).  {While these objects are not part of the primitive description of $\doi$, it follows from conjugacy of self-concordant barriers that $\nabla f = (\nabla f^*)^{-1}$ and $H(\cdot) = H^*(\nabla f(\cdot))^{-1}$ where $H^*$ is the Hessian of $f^*$ (see~\cite[Theorem~3.3.4]{Renegar_01}), and therefore one can work directly with $\nabla f$ and $H$ through the primitive objects $\nabla f^*$ and  $H^*$.} Of course, for standard logarithmically-homogeneous barriers $f^*$ such as $f^*(y)= -\sum_{i=1}^m \ln y_i$, where $y\in\bbR_{++}^m$ or $f^*(Y) = -\ln\det Y$, where $Y\in\bbS_{++}^p$ (and $m = p(p+1)/2$), its Fenchel conjugate of $f$ and the Hessian of $f$ are well-known. In addition, for many simple non-smooth functions $h^*$, e.g., $h^*(w) = \normt{w}_p$ ($p\in[1,+\infty]$), 
its Fenchel conjugate either is well-known or can be easily computed. Therefore, for many problems of interest implementing Algorithm~\ref{algo:DMD} via Algorithm~\ref{algo:FW_SC} is likely to be quite viable. 
 Second, note that the step-size $\gamma_k$ used in Algorithm~\ref{algo:DMD} is adaptive, and is different from a standard step-size that is monotone decreasing in $k$, e.g., $\gamma_k = O(1/\sqrt{k})$ or $\gamma_k = O(1/{k})$ (cf.~\cite{Nemi_79}). This poses difficulties in attempting to directly analyze Algorithm~\ref{algo:DMD} via standard approaches which involves using $D_{f^*}(y^k,y^*)$ as the potential function (see e.g.,~\cite{Bach_15}). Nevertheless, the convergence guarantee for 
$G_k$ derived from Algorithm~\ref{algo:FW_SC} enables us to analyze the converge of the dual optimality gap $d_k$ in Algorithm \ref{algo:DMD}. 


\section{Computational Experiments}\label{experiments}

In this section we present the results of some basic numerical experiments where we evaluate the performance of our generalized Frank-Wolfe method in Algorithm \ref{algo:FW_SC} on the Poisson image de-blurring problem with TV regularization (Application 1 in Section \ref{sec:applications}), and also on the PET problem (Application 2 of Section \ref{sec:applications}).

\subsection{First numerical experiment: Poisson image de-blurring with TV regularization}\label{sec:delurring}

We consider the Poisson image de-blurring problem with TV regularization as described in Application 1 of Section~\ref{sec:applications}, where the formulation was presented in equation~\eqref{eq:deblurring_TV}.  Observe that $f:u \mapsto -\textstyle\sum_{l=1}^N y_l\ln \big(u_l)$ is {neither Lipschitz nor $L$-smooth on the set $\{u\in\bbR^N: u = \rvA x, \ 0 \le x \le Me\}$}, and ${\rm TV}(\cdot)$ does not have an efficiently computable proximal operator, which prevents most standard first-order methods from being applicable to tackle~\eqref{eq:deblurring_TV}. As a result, very few methods have been proposed to solve~\eqref{eq:deblurring_TV} in the literature. In~\cite{Dey_06} an ad-hoc expectation-maximization (EM) method was proposed to solve~\eqref{eq:deblurring_TV}, however the method is not guaranteed to converge due to the non-smoothness of the function ${\rm TV}(\cdot)$ (see e.g.,~\cite{Pierro_95}). Both~\cite{Harmany_12} and~\cite{Chambolle_18} proposed methods to solve a ``perturbed'' version of~\eqref{eq:deblurring_TV} by adding a small offset to each $\ln(\cdot)$ term, which of course, compromises the original objective function $\barF(x)$ near $x=0$. (Such a perturbed version is not needed for the theoretical analysis in \cite{Chambolle_18}, but seems to be used to improve practical performance.)  Using the generalized Frank-Wolfe method of Algorithm \ref{algo:FW_SC}, we are able to directly solve formulation \eqref{eq:deblurring_TV}, the details of which we now describe.

\subsubsection{Implementation of generalized Frank-Wolfe method for solving \eqref{eq:deblurring_TV}}

We first re-describe the total variation function ${\rm TV}(x)$ by introducing some network definitions and terminology.  The TV function penalizes potential differences on the horizontal and vertical grid arcs of the standard $m \times n$ pixel grid.  Considering each pixel location as a node, let $\calV$ denote these nodes enumerated as $\calV = [N] =\{1, \ldots ,  N = m \times n\}$, and consider the undirected graph $\calG = (\calV,\calE)$ where $\calE$ is the set of horizontal and vertical edges of the grid.  These horizontal and vertical edges can be described as $\calE_\rmh$ and $\calE_\rmv$, respectively, where 
\begin{align*}
\calE_\rmh :=  \{\{l,l+1\}: l \in[N],\;\; l \!\!\! \mod n \ne 0\}\quad \mbox{and}\quad \calE_\rmv :=  \{\{l,l+n\}: l \in[N-m] \} \ .
\end{align*}
With this notation we have
\begin{equation}
{\rm TV}(x) = \textstyle{\sum}_{\{i,j\}\in \calE} \;\abst{x_i-x_j} \ . \label{eq:TV_abs}
\end{equation}
Based on $\calG$, we define the directed graph $\tilcalG = (\calV,\tilcalE)$ where $\tilcalE$ is obtained by replacing each (undirected) edge in  $\calE$ with two directed edges of opposite directions. Then from~\eqref{eq:TV_abs} we have 
\begin{align}
{\rm TV}(x) = \textstyle\sum_{(i,j)\in \tilcalE}\;\; \max\{x_i-x_j,0\} = \min&\;\; e^\top r\nn\\
\st& \;\;r_{ij}\ge x_i-x_j \ ,\;\; \forall\;(i,j) \in \tilcalE  \\
= \min & \;\; e^\top r\;\;\nn\\
\st& \;\;r\ge B^\top x \ , \label{eq:TV_LP} 
\end{align}
where $B$ is the node-arc incidence matrix of the network $\tilcalE$. 

The Frank-Wolfe subproblem~\eqref{subcbday} associated with~\eqref{eq:deblurring_TV} is:
\begin{align}
{\min}_{x\in\bbR^N}&\;\; \lranglet{\nabla \bar{f}(x^k)}{x} + \textstyle(\sum_{l=1}^{N} a_l)^\top x + \lambda {\rm TV}(x)\qquad \st\;\; 0\le  x\le M e \ , \label{eq:box_TV_LO}
\end{align}
where 
\begin{equation}
\barf(x): = -\textstyle\sum_{l=1}^{N} y_l\ln(a_l^\top x)  \ .\label{eq:barf_TV}
\end{equation} Based on~\eqref{eq:TV_LP}, we can rewrite~\eqref{eq:box_TV_LO} as the following linear optimization problem: 
\begin{align}
\min_{(x,r)\in\bbR^N\times \bbR^{2\abst{\calE}}}&\;\; \textstyle\lranglet{\nabla \bar{f}(x^k)}{x} + (\sum_{l=1}^{N} a_l)^\top x + \lambda e^\top r \nn\\
\st&\;\; 0\le  x\le M e, \;\;r\ge B^\top x \ . \label{eq:box_TV_LO2}
\end{align} This linear problem can be solved as a linear program (LP) using available software, or as a constrained dual network flow problem using available network flow software.  In our implementation we solved \eqref{eq:box_TV_LO2} using the primal simplex method in Gurobi \cite{gurobi}. Note that in~\eqref{eq:box_TV_LO2}, the variable $r$ in the constraint set appears to be unbounded. However, from the form of the objective function and the definition of $B$, it is easy to see that any optimal solution $(x^*,r^*)$ must satisfy $0\le r^*\le Me$, and hence~\eqref{eq:box_TV_LO2} always admits an optimal solution.  


Given the representation of $\TV(\cdot)$ in~\eqref{eq:TV_LP}, 
we can also rewrite~\eqref{eq:deblurring_TV} as 
\begin{align}
\min_{x\in\bbR^N, \;\;r\in\bbR^{2\abst{\calE}}}&\;\; \barf(x) + \textstyle(\sum_{l=1}^{N} a_l)^\top x + \lambda e^\top r \nn\\
\st&\;\; 0\le  x\le M e, \;\;r\ge B^\top x. \label{eq:box_TV2}
\end{align}
(Note that the Frank-Wolfe sub-problem~\eqref{subcbday} associated with~\eqref{eq:box_TV2} is 
the same as that associated with~\eqref{eq:deblurring_TV}, which is shown in~\eqref{eq:box_TV_LO2}.) 
In the following, we will apply our Frank-Wolfe method to solve the reformulated problem~\eqref{eq:box_TV2}. The advantage of~\eqref{eq:box_TV2} lies in that its structure yields an efficient procedure for an exact line-search to compute the step-size $\alpha_k$. Specifically, given $(x^k,r^k)$, let $(v^k,w^k)$ be an optimal solution of~\eqref{eq:box_TV_LO2}; then the exact line-search problem using \eqref{eq:box_TV2} is:
\begin{align}
\alpha_k = {\argmin}_{\alpha\in [0,1]} \; \barf(x^k + \alpha(v^k - w^k)) + \alpha \big(\textstyle(\sum_{l=1}^{N} a_l)^\top (v^k - x^k) + \lambda e^\top (w^k - r^k)\big) \ . \label{lulup}
\end{align}
The detailed description of the exact line-search procedure for problems of the format \eqref{lulup} are presented in Appendix~\ref{app:line_search}.  Henceforth, we denote our generalized Frank-Wolfe method for Poisson de-blurring with the adaptive step-size in~\eqref{eq:step_size_SC} as {\tt FW-Adapt}, and we denote our Frank-Wolfe method with exact line-search of \eqref{lulup} as {\tt FW-Exact}. Note that since in each iteration, {\tt FW-Exact} makes no less progress than {\tt FW-Adapt} in terms of the objective value, the computational guarantees in Theorem~\ref{thm:SC} (which are proved for {\tt FW-Adapt}) also apply to {\tt FW-Exact}. 

\subsubsection{Results}\label{sec:res_deblurring}

We tested {\tt FW-Adapt} and {\tt FW-Exact} on the Shepp-Logan phantom image~\cite{Shepp_74} of size $100\times 100$ (hence $N = 10,000$).  The true image $X$ is shown in Figure~\ref{fig:clean_blur_img}(a); this image acts as a 2D slice of a representative 3D human brain image, and is a standard test image in image reconstruction algorithms. We generated the blurred noisy image $Y$ shown in Figure~\ref{fig:clean_blur_img}(b) using the methodology exactly as described in Section~\ref{sec:applications}. For both {\tt FW-Adapt} and {\tt FW-Exact}, we chose the starting point $x^0 = {\sf vec}(Y)$, and we set $\lambda = 0.01$. In order to have a accurate computation of optimality gaps, we used CVXPY~\cite{Diamond_16} to (approximately) find the optimal objective value $\barF^*$ of~\eqref{eq:deblurring_TV}. All computations were performed in Python 3.8 on a laptop computer. 

\begin{figure}[t!]\centering
\subfloat[True image $X$]{\includegraphics[width=.3\linewidth]{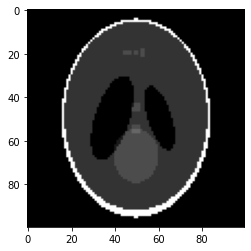}}
\subfloat[Noisy image $Y$]{\includegraphics[width=.3\linewidth]{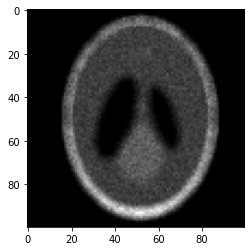}} 
\caption{True and noisy $100\times 100$ versions of the Shepp-Logan phantom image.}\label{fig:clean_blur_img} 
\end{figure}

Figures \ref{fig:deblur_gap}(a) and \ref{fig:deblur_gap}(b) show (in log-log scale) the empirical optimality gaps $\barF(x^k) - \barF^*$ obtained by {\tt FW-Adapt} and {\tt FW-Exact} both as a function of the wall-clock time (in seconds) and as a function of the iteration counter $k$, respectively. From the figure we observe that {\tt FW-Exact} converges faster than {\tt FW-Adapt}, although the difference between the empirical optimality gaps of these two methods gradually lessens. This is expected since with exact line-search, {\tt FW-Exact} can take a potentially larger step at each iteration than {\tt FW-Adapt}, and likely therefore has faster convergence. The recovered images computed using {\tt FW-Adapt} and {\tt FW-Exact} are shown in Figure~\ref{fig:Box_TV_res_img}(a) and Figure~\ref{fig:Box_TV_res_img}(b), respectively. We observe that the image recovered by {\tt FW-Adapt} has similar but slightly inferior quality compared to that recovered by {\tt FW-Exact}. This is consistent with the algorithms' performance shown Figure~\ref{fig:deblur_gap}, as {\tt FW-Adapt} has a larger empirical optimality gap at termination compared to that of {\tt FW-Exact}.  

\begin{figure}[t]
\subfloat[Optimality gap versus time (in seconds)]{\includegraphics[width=.48\linewidth]{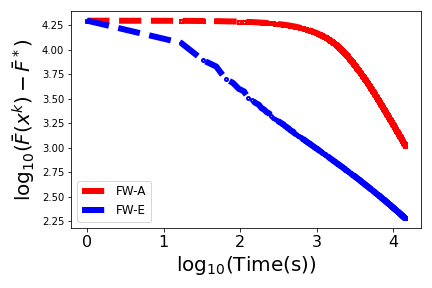}} \hfill 
\subfloat[Optimality gap versus iterations]{\includegraphics[width=.48\linewidth]{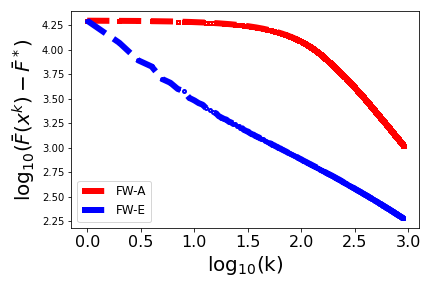}}
\caption{Comparison of empirical optimality gaps of {\tt FW-Adapt} ({\tt FW-A}) and {\tt FW-Exact} ({\tt FW-E}) for image recovery of the Shepp-Logan phantom image~\cite{Shepp_74} of size $100\times 100$.}\label{fig:deblur_gap}
\end{figure}

\begin{figure}[t!]\centering
\subfloat[Recovered image: {\tt FW-Adapt}]{\includegraphics[width=.3\linewidth]{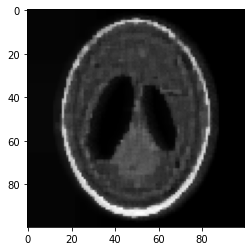}}
\subfloat[Recovered image: {\tt FW-Exact}]{\includegraphics[width=.3\linewidth]{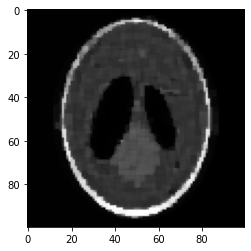}}
\caption{Recovered images computed using {\tt FW-Adapt} and {\tt FW-Exact}.}\label{fig:Box_TV_res_img} 
\end{figure}

\subsection{Second numerical experiment: positron emission tomography}\label{sec:PET}

We consider the positron emission tomography (PET) problem as described in Application 2 of Section~\ref{sec:applications}, where the formulation was presented in equation \eqref{eq:PET_final}.  We generated artificial data instances of this problem according to the following data generation process. Because the events emitted from each voxel $i$  can only be detected by a small proportion of bins, it follows that the probability matrix $P$ should be highly sparse.  We chose a sparsity value of $5\%$. Given the number of voxels $n$ and the number of bins $m$, for each $i \in [n]$ we randomly chose (without replacement) a subset of $[m]$ denoted by $\calJ_i$, such that $\abst{\calJ_i}=\floor{m/20}$.  Next, for all $j\in\calJ_i$ we then generated \iid entries $\bar p_{ij} \sim U[0,1]$ and normalized the values to obtain $p_{ij} := \barp_{ij}/\sum_{j'\in\calJ_i} \bar p_{ij'}$. For all $j\not\in \calJ_i$ we set $p_{ij}=0$. We generated the mean values $x_i$ for $i \in [n]$ by first generating \iid values $\bar x_i \sim N(100,9)$ 
and then set $x_i = \abst{\barx_i}$ for $i\in[n]$.  We then simulated the event counts $X_i$ at each voxel $i$ by generating $X_i \sim {\sf Poisson}(x_i)$ for $i\in[n]$.  Finally, using $P$ and $\{X_i\}_{i\in[n]}$, we generated the number of observed events $Y_j$ detected at bin $j$ by independently generating  values $\tilde Y_j \sim {\sf Poisson}(y_j)$ with $y_j := \textstyle\sum_{i=1}^n p_{ij} X_i$ for $j\in[m]$. Since $Y_j\in\{0,1,2,\ldots\}$, by omitting bins for which $Y_j = 0$ we ensure that $Y_j\ge 1$ for all $j\in[m]$ in the PET problem~\eqref{eq:PET_final}. 

\subsubsection{Comparison of Algorithms}\label{sec:benchmark}
We solved instances of the PET problem \eqref{eq:PET_final} using the following five algorithms/variants:
\begin{itemize}
	\item {\tt FW-Adapt} -- our generalized Frank-Wolfe method in Algorithm~\ref{algo:FW_SC} with the adaptive step-size as stated in~\eqref{eq:step_size_SC},
	\item {\tt FW-Exact} -- our generalized Frank-Wolfe method in Algorithm~\ref{algo:FW_SC} with the adaptive step-size \eqref{eq:step_size_SC} replaced by an exact line-search as described in detail in Appendix~\ref{app:line_search},
	\item {\tt RSGM-Fixed} -- relatively smooth gradient method with fixed step-size~\cite{Bauschke_17,Lu_18}, 
	\item {\tt RSGM-LSBack} -- relatively smooth gradient method with backtracking line-search~\cite{Ston_20},
	\item {\tt EM} -- a simple algorithm developed by Cover in 1984 specifically for a problem that is equivalent to the normalized PET problem ~\cite{Cover_84}.  
\end{itemize}We excluded mirror descent from our computational comparisons because the sparsity of $P$ violates the basic assumption needed to apply the mirror descent (MD) method to \eqref{eq:PET_final} (see e.g.,~\cite{BenTal_01}).  We now review the relevant details of the three algorithms {\tt RSGM-F}, {\tt RSGM-LS} and {\tt EM}. 

\vspace{1ex}
\noindent 1.\ {\tt RSGM-Fixed}~\cite{Bauschke_17,Lu_18}. Although the objective function $L$ of~\eqref{eq:PET_final} is differentiable on its domain, its gradient $\nabla L$ is not Lipschitz on the constraint set $\Delta_n$.  
Therefore standard gradient methods (or accelerated versions)~\cite{Nest_13} are not applicable.  As a remedy for this situation, we can use the relatively-smooth gradient method \cite{Bauschke_17,Lu_18} to solve~\eqref{eq:PET_final}, which is designed in particular for problems whose objective functions have certain types of structured non-smooth gradients.  Indeed, as shown in Bauschke et al.~\cite{Bauschke_17}, $L$ is $\bar Y$-smooth relative to the reference function 
\begin{equation}
r(z) := -\textstyle\sum_{i=1}^n \ln (z_i) \ \mbox{for} \  z\in\bbR^n_{++}:=\{x\in\bbR^n:x>0\} \ , 
\end{equation}
where $\bar Y := \sum_{j=1}^m Y_j$.  Specifically, this means that
\begin{equation}
\nabla^2 L(z) \preceq \bar Y \nabla^2 r(z) \quad \forall\,z\in\bbR^n_{++}\;  \ . 
\end{equation}
Algorithm \ref{algo:RSGM} describes RSGM specified to solve the PET problem \eqref{eq:PET_final} using the reference function $r$ with relative-smoothness parameter $\bar Y$, where  $\relint\Delta_n$ denotes the relative interior of $\Delta_n$ in the Input, and $D_r(\cdot,\cdot)$ denotes the Bregman divergence in~\eqref{eq:equiv_BP} which is defined similarly as in~\eqref{eq:Bregman}. 
We set the step-size $\alpha_k =1/\barY$ for all $k\ge 0$ in the fixed-step-size version of the method. 
 Regarding the sub-problem \eqref{eq:equiv_BP} that needs to be solved at each iteration, its optimal solution is unique and lies in $\relint\Delta_n$ since the reference function $r$ is Legendre with domain $\bbR_{++}^n$ (see~\cite[Section~2.1]{Bauschke_17} for details). Therefore we have $z^k\in\ri\Delta_n$ for all $k\ge 0$. 
To efficiently solve \eqref{eq:equiv_BP}, we used the approach in~\cite[pp.\ 341-342]{Lu_18}, which reduces \eqref{eq:equiv_BP} to finding the unique positive root of a strictly decreasing univariate function on $(0,+\infty)$. 

\begin{algorithm}[t!]
\caption{RSGM for solving the PET problem~\eqref{eq:PET_final}} \label{algo:RSGM}
\begin{algorithmic}
\State {\bf Input}: Starting point $z^0\in\relint\Delta_n := \{z>0:\sum_{i=1}^n z_i = 1\}$. 
\State {\bf At iteration $k$:} 
\begin{enumerate}
{\setlength\itemindent{10pt} \item {Compute  $\nabla L(z^k) = \sum_{j=1}^m (Y_j/\lranglet{p_j}{z^k}) p_j$, where $p_j$ 
is the vector corresponding to the $j$-th column of $P$ for $j\in[m]$\ .}} 
{\setlength\itemindent{10pt} \item {Choose step-size $\alpha_k>0$ and compute} 
\begin{align}
z^{k+1} &=  {\argmax}_{z\in\Delta_n}\; \lranglet{\nabla L(z^k)}{z} - \alpha_k^{-1} D_r(z,z^k) \ .
\label{eq:equiv_BP}
\end{align}
}
\end{enumerate}
\end{algorithmic}
\end{algorithm}

\vspace{1ex}
\noindent  2.\ {\tt RSGM-LSBack}~\cite{Ston_20}.
This method is a version of the relatively smooth gradient method shown in Algorithm~\ref{algo:RSGM}, with the extension that a backtracking line-search procedure is employed to compute the local relative-smoothness parameter $\bar Y_k$ at $z^k$ and then the step-size is chosen as $\alpha_k = 1/\bar Y_k$ at each iteration. (The details of this procedure can be found in~\cite[Algorithm~1]{Ston_20}.) Note that depending on the location of $z^k$, $\bar Y_k$ may be (significantly) smaller than the (global) relative-smoothness parameter $\barY$. 

\vspace{1ex}
\noindent 3.\ {\tt EM}~\cite{Cover_84}. 
This surprisingly simple algorithm was developed by Cover in 1984 specifically for a problem that is equivalent to the following normalized PET problem (see \cite{Cover_84}):
\begin{equation}
\max \; \barL(z) := \textstyle\sum_{j=1}^m \barY_j\ln \big(\sum_{i=1}^n p_{ij} z_i\big)\quad \st\;\; {z\in\Delta_n} \ ,
\end{equation}
where $\barY_j:= Y_j/\sum_{j'=1}^m Y_{j'}>0$ for all $j\in[m]$ (recall that we have assumed without loss of generality that $Y_j\ge 1$ for all $j\in[m]$). 
The method starts with $z^0\in \ri\Delta_n$ and at each iteration $k$ updates $z^k$ as follows:
\begin{equation}
z^{k+1}_i := z^k_i \nabla_i \barL(z^k) =  \sum_{j=1}^m \barY_j \frac{p_{ij} z^k_i }{\sum_{l=1}^n p_{lj} z_l^k} \ , \quad \forall\,i\in[n] \ . \label{eq:EM}
\end{equation} 
Note that since $\barY_j>0$ for all $j\in[m]$, we easily see that $z^k\in\ri\Delta_n$ for all $k\ge 0$.


\subsubsection{Results}

We report results for problems of dimensions $n=m=1,000$, as we observed that the algorithms' relative performance were not particularly sensitive to the dimensions $m$ and $n$.  We ran the algorithms using two different choices of starting points: (i) $z^0 = z^{\mathrm{bd}}\in \relint\Delta_n$ chosen very close to $\bdry\bbR^n_+$ (the boundary of $\bbR^n_+$), and (ii) $z^0 = z^{\mathrm{ct}}:=\tfrac{1}{n}e$ which is the barycenter of $\Delta_n$.  To determine $z^{\mathrm{bd}}$ we first used a greedy heuristic to determine a low- (or lowest-)cardinality index set $\calI\subseteq [n]$ 
 for which $\sum_{i\in\calI} \, p_j>0$, where $p_j$ is the $j$-th column of $P$. Define $\bar\delta := 10^{-6}/n$, and we then set 
$$z^{\mathrm{bd}}_i := \left\{ \begin{array}{ll} \bar\delta & \mbox{for} \ \ i\not\in \calI \ ,  \\   (1 - (n-\abst{\calI})\bar\delta)/\abst{\calI} \ \ & \mbox{for} \ \ i\in \calI \ .  \end{array} \right. $$ Note that this ensures $z^{\mathrm{bd}}\in\ri\Delta_n$. Similar to Section~\eqref{sec:delurring}, we used CVXPY~\cite{Diamond_16} to (approximately) compute the optimal objective function value $L^*$ of~\eqref{eq:PET_final} in order to accurately report optimality gaps.  Again, all computations were performed in Python 3.8 on a laptop computer.

Figures~\ref{fig:exp_res1} and \ref{fig:exp_res2} show plots of the empirical optimality gaps $L(z^k) -L^*$ of all five methods with the starting points $z^{\mathrm{bd}}$ and $z^{\mathrm{ct}}$, respectively. From Figure~\ref{fig:exp_res1} we observe the following: 
\begin{enumerate}[label = (\roman*)]
\item The relatively smooth gradient methods, namely {\tt RSGM-Fixed} and {\tt RSGM-LSBack}, make very little progress in reducing the empirical optimality gap. For {\tt RSGM-Fixed}, this is because the relative smoothness parameter $\barY = \sum_{j=1}^m Y_j$ is typically very large, implying that the step-size $1/\barY$  is very small. Since the starting point $z^{\mathrm{bd}}$ is very close to $\bdry\bbR^n_+$, where the local relative smoothness parameter of the objective function $L$ is close to $\barY$, {\tt RSGM-LSBack} exhibits similar behavior to {\tt RSGM-Fixed}.  (In other words, the backtracking line-search does not help much in this case.)

\item The two versions of our generalized Frank-Wolfe methods, namely {\tt FW-Adapt} and {\tt FW-Exact}, outperform the relatively smooth gradient methods. In addition, {\tt FW-Exact} converges faster than {\tt FW-Adapt} initially, but when close to the optimum (for example, when the empirical optimality gap falls below 10), these two methods have the same convergence behavior. Indeed, this observation also appears on the Poisson image de-blurring problem with TV-regularization (see Section~\ref{sec:res_deblurring}).

\item The {\tt EM} algorithm outperforms all the other methods which is rather surprising at first glance. (And in fact it is unknown from the literature whether the method has any type of non-asymptotic convergence guarantee.) However, we note that this method, which uses a multiplicative form of update (see equation~\eqref{eq:EM}), is specifically designed for problems with the PET problem structure in~\eqref{eq:PET_final}, and it is not clear how to suitably generalize this method to more general problems of the form $\poi$ such as the Poisson image de-blurring problem with TV-regularization in~\eqref{eq:deblurring_TV}. 
\end{enumerate}

Figure~\ref{fig:exp_res2} shows the performance of the five algorithms when $z^0 = z^{\mathrm{ct}}$ is the barycenter.  We observe that the performance of the methods is mostly similar to when started at $z^{\mathrm{bd}}$ except for one significant difference, namely that {\tt RSGM-LSBack} exhibits significantly faster convergence compared to when started at $z^{\mathrm{ct}}$. Indeed, {\tt RSGM-LSBack} outperforms all the other general purpose methods, namely {\tt RSGM-Fixed}, {\tt FW-Adapt} and {\tt FW-Exact}. This is likely because on the ``central region'' of $\Delta_n$ the local relative smoothness parameter of $L$ is probably much smaller than the global bound $\barY$, and therefore {\tt RSGM-LSBack} is able to take a much larger steps.  This also indicates that the convergence behavior of {\tt RSGM-LSBack} is likely to be sensitive to the choice of starting point, which is not the case for the other methods (including our generalized Frank-Wolfe methods).  



\begin{figure}[t!]\centering
\subfloat[Time plot: $n=1000$, $m = 1000$]{\includegraphics[width=.48\linewidth]{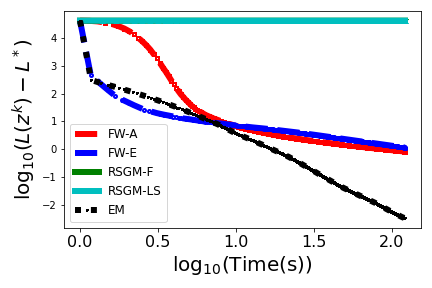}}\hfill
\subfloat[Iteration plot: $n=1000$, $m = 1000$]{\includegraphics[width=.48\linewidth]{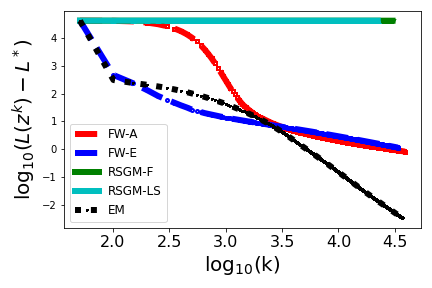}}
\caption{Comparison of optimality gaps of {\tt FW-Adapt} ({\tt FW-A}), {\tt FW-Exact} ({\tt FW-E}), {\tt RSGM-Fixed} ({\tt RSGM-F}), {\tt RSGM-LSBack} ({\tt RSGM-LS}) and {\tt EM}, with $z^0 = z^{\mathrm{bd}}$.}\label{fig:exp_res1}
\end{figure}

\begin{figure}[t!]\centering
\subfloat[Time plot: $n=1000$, $m = 1000$]{\includegraphics[width=.48\linewidth]{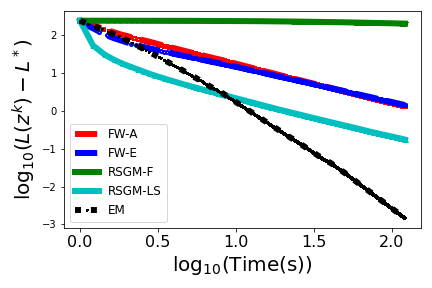}}\hfill
\subfloat[Iteration plot: $n=1000$, $m = 1000$]{\includegraphics[width=.48\linewidth]{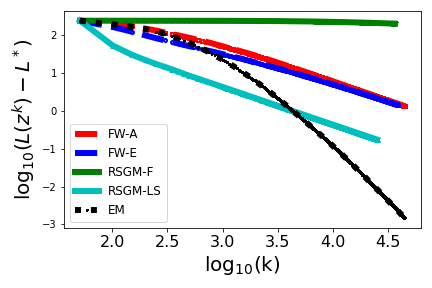}}
\caption{Comparison of optimality gaps of {\tt FW-Adapt} ({\tt FW-A}), {\tt FW-Exact} ({\tt FW-E}), {\tt RSGM-Fixed} ({\tt RSGM-F}), {\tt RSGM-LSBack} ({\tt RSGM-LS}) and {\tt EM}, with $z^0 = z^{\mathrm{ct}}$.}\label{fig:exp_res2}
\end{figure}

\appendix

\section{Complexity of computing $D_k$ in the $D$-optimal design problem}\label{app:D_k}


Let the objective function $\barf$ be defined in~\eqref{eq:Dopt2}, and denote its Hessian at $x\in\dom \barf$ as $\barH(x)$. Also denote $C:= [a_1\; \ldots\; a_m]\in\bbR^{n\times m}$ and  $X := \diag(x)$. From standard results (e.g.,~\cite[Proposition 2.2]{Lu_18}),  we know that $\nabla_i f(x) = - Q(x)_{ii}$  and $\barH(x)_{ii} = Q(x)_{ii}^2$ for $i\in[m]$, where $Q(x):= C^\top (CXC^\top)^{-1} C$ and hence $Q(x)_{ii} = a_i^\top (CXC^\top)^{-1} a_i$. 
Since the constraint set in~\eqref{eq:Dopt2} is $\Delta_m$, we can choose $v^k = e^{i_k}$, where $i_k\in\argmin_{i\in[m]} \nabla_i f(x^k)$ and $e^i$ denotes the $i$-th standard coordinate vector. Let $Q^k := Q(x^k)$.  Now, note that we can write 
\begin{align*}
D_k^2 = \lranglet{\barH (x^k)(x^k - v^k)}{x^k - v^k} &= \lranglet{\barH (x^k)x^k}{x^k} - 2\lranglet{\barH (x^k)x^k}{v^k} + \lranglet{\barH (x^k)v^k}{v^k}\\
&\eqa -\lranglet{\nabla \barf (x^k)}{x^k} + 2\lranglet{\nabla \barf (x^k)}{e^{i_k}} + \lranglet{\barH (x^k)e^{i_k}}{e^{i_k}}\\
&\eqb n + 2\nabla_{i_k} \barf (x^k) + \barH (x^k)_{i_k,i_k}\\
&= n - 2 Q^k_{i_k,i_k} + \big(Q^k_{i_k,i_k}\big)^2 \ .
\end{align*}
where in (a) we use~\ref{item:Hessian_grad} and $v^k = e^{i_k}$ and in (b) we use~\ref{item:grad_identity} and that the complexity parameter of $\barf$ is $n$. Therefore, the computational complexity of $D_k$ is the same as that of $Q^k_{i_k,i_k} = a_{i_k}^\top B^k a_{i_k}$, where $B^k := (CX^kC^\top)^{-1}$. If $k=0$, this can be computed in $O(mn^2 + n^3)$ time. 
In addition, since we have $\alpha_k < 1$ for all $k\ge 0$ (because $e^i\not\in\dom \barf$ for any $i\in[m]$), the following holds:
\begin{align*}
Q^k_{i_{k+1},i_{k+1}} &= a_{i_{k+1}}^\top \big((1-\alpha_k)CX^{k}C^\top + \alpha_k a_{i_k}a_{i_k}^\top\big)^{-1} a_{i_{k+1}}\\
&\eqa (1-\alpha_k)^{-1}a_{i_{k+1}}^\top \left(B^k - \frac{\beta_k B^k a_{i_k}a_{i_k}^\top B^k}{1+\beta_k a_{i_k}^\top B^k a_{i_k}}\right) a_{i_{k+1}}\qquad [\mbox{where }\beta_k:= \alpha_k/(1-\alpha_k)],\\
&= (1-\alpha_k)^{-1} \left(a_{i_{k+1}}^\top B^k a_{i_{k+1}} - \frac{ (a_{i_{k+1}}^\top B^k a_{i_k})^2}{\beta_k^{-1}+Q^k_{i_k,i_k}}\right),
\end{align*}
where (a) follows from the Inverse Matrix Update formula \cite{hager}.
Therefore, given $B^k$ and $Q^k_{i_k,i_k}$, $Q^k_{i_{k+1},i_{k+1}}$ can be computed in $O(n^2)$ time, for all $k\ge 0$. 

\section{Proof of Proposition~\ref{lem:omega_conj}}\label{app:proof_omega_conj}
Note that $(\omega^*)'(s) = s/(1+s)>0$ for any $s> 0$, so $\omega^*$ is strictly increasing on $[0,+\infty)$. In addition, 
$0\le s\le 1/2$ implies that  $(\omega^*)'(s)\ge \tfrac{2}{3}s$, whereby $$\omega^*(s) = \omega^*(0) + \int_{0}^s (\omega^*)'(t) dt \ge \int_{0}^s \tfrac{2}{3} t \ dt = s^2/3 \ \ \mbox{for~} s\in [0,1/2] \ . $$ 
Define the linear function $l(s):= (1-2\ln(1.5))s$, and notice that $l(s)=\omega^*(s)$ at $s=0$ and $s=1/2$.  Therefore from the convexity of $\omega^*$ it follows for $s \ge 1/2$ that
$$\omega^*(s) \ge l(s) = (1-2\ln(1.5))s \ge s/5.3 \ . $$ \qed

\section{Proof of Proposition~\ref{lem:Karmarkar}}\label{app:proof_Karmarkar}
For any $s\in(-1,1)$, we have 
\begin{align*}
\ln(1+s) &= s - \frac{s^2}{2} + \frac{s^3}{3} - \frac{s^4}{4} + \cdots\\
&\ge s - \frac{\abs{s}^2}{2} - \frac{\abs{s}^3}{2} - \frac{\abs{s}^4}{2} - \cdots\\
&= s - \frac{\abs{s}^2}{2} \left(1 + \abs{s} + \abs{s}^2 + \cdots\right)\\
&= s - \frac{\abs{s}^2}{2(1-\abs{s})} \ . 
\end{align*}\qed

\section{Proof of inequality~\eqref{eq:bound_delta_0}}\label{app:proof_bound_delta_0}
As shown in Khachiyan~\cite[Lemma~3]{khachiyan1996rounding}, with $p_0 = (1/m)e$ we have\begin{equation}\label{late}
\delta_0\le n\ln(1+\varepsilon_0)  \quad \mbox{where}\;\; \varepsilon_0:= \frac{m}{n}\,\max_{i\in[m]} a_i^\top\big(\textstyle\sum_{j=1}^m a_ja_j^\top\big)^{-1} a_i - 1 \ . 
\end{equation}
Define $C:= [a_1\; \ldots\; a_m]\in\bbR^{n\times m}$ 
and $P:= C^\top\big(CC^\top\big)^{-1}C$, and it holds for all $i\in[m]$ that:
\begin{equation}
a_i^\top\big(\textstyle\sum_{j=1}^m a_ja_j^\top\big)^{-1} a_i = (e^i)^\top C^\top\big(CC^\top\big)^{-1}Ce^i = \normt{Pe^i}_2^2  \le \normt{e^i}_2^2 = 1 \ ,
\end{equation}
where $e^i$ denotes the $i$-th standard coordinate vector, and the second equality and the inequality above follow since $P$ is a projection matrix.  As a result, $\varepsilon_0\le m/n-1$ and the proof follows using the left inequality in \eqref{late}. \qed

\section{Proof of Proposition \ref{greatrun}}\label{holdenwood}

Note that \eqref{eq:rate_d_j} follows easily from
\begin{equation}
\frac{1}{d_{j+1}   } \ge \frac{1}{d_j   } + \frac{1}{  M} \quad\forall\,j\ge 0 \ , \label{eq:recip_ineq} 
\end{equation}
and so let us prove \eqref{eq:recip_ineq}.  It follows from the hypotheses of Proposition \ref{greatrun} that:
\begin{align}
d_{j+1} \le d_j - \frac{g_j^2}{ M}\le d_j - \frac{d_j^2}{ M} \ ,\label{eq:basic_ineq}
\end{align}
and dividing both sides of~\eqref{eq:basic_ineq} by $d_jd_{j+1}$ and rearranging yields
\begin{align*}
\frac{1}{d_{j+1}} \ge \frac{1}{d_j} + \frac{d_j}{ M d_jd_{j+1}}\ge \frac{1}{d_j}\left(1+ \frac{d_j   }{ M} \right)  = \frac{1}{d_j   } + \frac{1}{ M} \ ,
\end{align*}
where the second inequality uses $d_j\ge d_{j+1}$.  

We next prove \eqref{eq:rate_g_j}. The result is trivally true for $j=0$, so let us consider $j \ge 1$. For convenience define $\barg_j:= \min\{g_0, \ldots, g_j\}$, and consider some $l$ satisfying some $0 < l\le j$. We then have
\begin{align*}
\frac{M}{l }>d_l = d_{j+1} + \sum_{i=l}^j d_i - d_{i+1} \ge d_{j+1} + \sum_{i=l}^j \frac{g_i^2}{ M} \ge (j-l+1) \frac{\barg_j^2}{ M} \ , 
\end{align*}
where the first inequality uses \eqref{eq:rate_d_j} and the second inequality is from the first hypothesis of the proposition.  
If we take $l = \ceil{j/2}$, then $j/2\le l\le j/2+1$. Therefore 
\begin{align*}
\barg_j^{2} < \frac{ M^{2}}{l (j-l+1)} \le \frac{  M^{2}}{  (j/2)^2} =  \left(\frac{2M}{j}\right)^{2} \ ,
\end{align*}
which then implies \eqref{eq:rate_g_j}. \qed

\section{Exact line-search procedure to compute $\alpha_k$}\label{app:line_search}


Let us define 
\begin{align}\barF(x):= -\textstyle\sum_{i\in[m]} y_i\ln(a_i^\top x) + \xi^\top x \ , \label{flulu}\end{align} where $y_i\in\{1,2,\ldots\}$ and $a_i\in\bbR^n$ for $i\in[m]$ and $\xi \in\bbR^n$, and we wish to solve $\min_{x\in\calX}\, \barF(x)$. Note that the optimization problems \eqref{eq:box_TV2} and \eqref{eq:PET_final} both conveniently fit into the format \eqref{flulu}. 
 Given $x\in\calX\cap\dom \barF$, let $d\in\bbR^n$ be a descent direction at $x$ (namely $\lranglet{\nabla \barF(x)}{d}<0$) such that $x+ d \in\calX$. The exact line-search problem involves finding 
\begin{equation}
\alpha^*:= {\argmin}_{\alpha\in[0,1]} \; [\zeta(\alpha):= \barF(x+\alpha d)] \ . 
\end{equation}
Define 
\begin{align}
\calI_+:=\{i\in[m]: a_i^\top d > 0\}\quad \mbox{and} \quad \calI_-:=\{i\in[m]: a_i^\top d < 0\} \ , \label{eq:I_+} 
\end{align}
so that 
\begin{align*}
\zeta(\alpha)\eqcst  -\textstyle\sum_{i\in\calI_+} y_i\ln(a_i^\top x + \alpha a_i^\top d) -\textstyle\sum_{i\in\calI_-} y_i\ln(a_i^\top x + \alpha a_i^\top d) + \alpha \xi^\top d \ ,
\end{align*}
where 
$\eqcst$ denotes equality up to a constant independent of $\alpha$. Consequently, we have 
\begin{align*}
\zeta'(\alpha) &= -\sum_{i\in\calI_+} \frac{y_ia_i^\top d}{a_i^\top x + \alpha a_i^\top d} -\sum_{i\in\calI_-} \frac{y_ia_i^\top d}{a_i^\top x + \alpha a_i^\top d} + \xi^\top d\\
&=  -\sum_{i\in\calI_+} \frac{y_i}{\gamma_i + \alpha} -\sum_{i\in\calI_-} \frac{y_i}{\gamma_i + \alpha} + \xi^\top d \ ,
\end{align*}
where $\gamma_i = a_i^\top x/a_i^\top d$  for $i\in\calI_+\cup \calI_-$. Clearly, we have $\gamma_i > 0$ for $i\in\calI_+$ and $\gamma_i < 0$ for $i\in\calI_-$. We have the following observations about $\zeta$ and $\zeta'$:
\begin{enumerate}[label = (\roman*)]
\item Domain: $\dom \zeta=\dom \zeta'= (\gamma_-,\gamma_+)$, where $\gamma_-:= -\min_{i\in\calI_+}\gamma_i<0$ and $\gamma_+:= -\max_{i\in\calI_-}\gamma_i>0$. 
\item Monotonicity: Since $\zeta$ is strictly convex, $\zeta'$ is strictly increasing on $(\gamma_-,\gamma_+)$. Also, $\zeta'(0) =   \lranglet{\nabla \barF(x)}{d} <0$. 
\item \label{item:behavior_alpha_+infty} Behavior as $\alpha\uparrow +\infty$: If $\calI_-\ne \emptyset$, then $\gamma_+ < +\infty$ and $\zeta'(\alpha)\uparrow +\infty$ as $\alpha\uparrow \gamma_+$; if $\calI_- =  \emptyset$, then $\gamma_+ = +\infty$ and  $\zeta'(\alpha)\uparrow \xi^\top d$ as $\alpha\uparrow +\infty$. 
\end{enumerate} 
Based on these observations, our exact line search procedure can be described as follows:
\begin{itemize}
\item $\calI_-=\emptyset$: We consider two cases, namely $\xi^\top d\le 0$ and $\xi^\top d> 0$. If $\xi^\top d\le 0$, then $\zeta'(\alpha)<0$ for all $\alpha\ge 0$, and hence $\alpha^* = 1$. Otherwise, we have $\xi^\top d> 0$, and hence let us first find $\alpha'$ as the unique solution of $\zeta'(\alpha) = 0$, namely $\alpha' = (\zeta')^{-1}(0)$. If $\alpha'>1$, then $\zeta'(\alpha)<0$ for all $\alpha\in [0,1]$ and hence $\alpha^*=1$; otherwise $\alpha^* = \alpha'$. In other words, $\alpha^* = \min\{\alpha',1\}$. 
\item  $\calI_-\ne \emptyset$: In this case, by observation~\ref{item:behavior_alpha_+infty}, we see that $\zeta'(\alpha) = 0$ must have a unique solution $\alpha'$ on $(0,\gamma_+)$. To find such a solution, we can first find $j$ as the largest integer such that $\zeta(\gamma_+ -10^{j})>0$, and then use Newton's method starting from $\gamma_+ -10^{j}$ or bisection over $[0,\gamma_+ -10^{j}]$ to find $\alpha'$. Once we find $\alpha'$, we can let $\alpha^* = \min\{\alpha',1\}$. 
\end{itemize}




%

\bibliographystyle{spmpsci}      

\bibliography{math_opt,stat_ref,mach_learn,GF-papers-nips-better}

\end{document}

%% file: preamble2.tex
\usepackage[mathscr]{eucal}
\usepackage{epsfig,epsf,psfrag}
\usepackage{amssymb,amsfonts,latexsym}
\usepackage{amsmath}
\usepackage{graphicx}
\usepackage{epstopdf}
\usepackage{bm,xcolor,url}
\usepackage{fixltx2e}
\usepackage{array}
\usepackage{verbatim}
\usepackage{algpseudocode, cite}
\usepackage{algorithm}
\usepackage{verbatim}
\usepackage{textcomp}
\usepackage{mathrsfs}
\usepackage[referable]{threeparttablex}
\usepackage{subfig}
\usepackage{amsthm}
\usepackage{enumerate}
\usepackage{footnote}
\usepackage{enumitem}
\usepackage{xr}
\usepackage{mathtools}
\catcode`~=11 \def\UrlSpecials{\do\~{\kern -.15em\lower .7ex\hbox{~}\kern .04em}} \catcode`~=13 

\allowdisplaybreaks[3]

\newcommand{\tnorm}[1]{{\left\vert\kern-0.25ex\left\vert\kern-0.25ex\left\vert #1 
    \right\vert\kern-0.25ex\right\vert\kern-0.25ex\right\vert}}
\newcommand{\tnormt}[1]{{\vert\kern-0.25ex\vert\kern-0.25ex\vert #1 
    \vert\kern-0.25ex\vert\kern-0.25ex\vert}}

\newcommand{\normt}[1]{\Vert#1\Vert}
\newcommand{\abst}[1]{\vert#1\vert}
\newcommand{\abs}[1]{\left\lvert#1\right\rvert}

\newcommand{\nn}{\nonumber}

\newcommand{\eqcst}{\stackrel{{\rm c}}{=}}
 
\newcommand{\nt}{\addtocounter{equation}{1}\tag{\theequation}} 

\newcommand{\dom}{\mathsf{dom}\,}

\newcommand{\inter}{\mathsf{int}\,}
\newcommand{\bdry}{\mathsf{bd}\,}
\newcommand{\diag}{\mathsf{diag}\,}
\newcommand{\cl}{\mathsf{cl}\,}
\newcommand{\ri}{\mathsf{ri}\,}

\newcommand{\relint}{\mathsf{ri}\,}

\newcommand{\cone}{\mathrm{cone}\,}

\newcommand{\ipt}{\lranglet}

\newcommand{\calB}{\mathcal{B}}

\newcommand{\calD}{\mathcal{D}}
\newcommand{\calE}{\mathcal{E}}

\newcommand{\calG}{\mathcal{G}}

\newcommand{\calI}{\mathcal{I}}
\newcommand{\calJ}{\mathcal{J}}
\newcommand{\calK}{\mathcal{K}}

\newcommand{\calP}{\mathcal{P}}
\newcommand{\calQ}{\mathcal{Q}}

\newcommand{\calS}{\mathcal{S}}

\newcommand{\calV}{\mathcal{V}}

\newcommand{\calX}{\mathcal{X}}

\newcommand{\tilcalE}{\widetilde{\calE}} 
 
\newcommand{\tilcalG}{\widetilde{\calG}}

\newcommand{\rmh}{\mathrm{h}}

\newcommand{\rmv}{\mathrm{v}}

\newcommand{\bbC}{\mathbb{C}}

\newcommand{\bbE}{\mathbb{E}}

\newcommand{\bbH}{\mathbb{H}}

\newcommand{\bbR}{\mathbb{R}}
\newcommand{\bbS}{\mathbb{S}}

\DeclareMathAlphabet{\mathbsf}{OT1}{cmss}{bx}{n}

\newcommand{\rvA}{\mathsf{A}}

\newcommand{\rvI}{\mathsf{I}}

\newcommand{\tilG}{\widetilde{G}}

\newcommand{\hatP}{\widehat{P}}

\newcommand{\tily}{\widetilde{y}}
\newcommand{\tilY}{\widetilde{Y}}

\newcommand{\barf}{\bar{f}}
\newcommand{\barg}{\bar{g}}
\newcommand{\barh}{\bar{h}}

\newcommand{\barp}{\bar{p}}

\newcommand{\barx}{\bar{x}}

\newcommand{\barD}{\bar{D}}

\newcommand{\barF}{\bar{F}}
\newcommand{\barG}{\bar{G}}
\newcommand{\barH}{\bar{H}}

\newcommand{\barL}{\bar{L}}

\newcommand{\barY}{\bar{Y}}

\newcommand{\iid}{i.i.d.\ }

\newcommand{\ceil}[1]{\lceil{#1}\rceil}
\newcommand{\floor}[1]{\lfloor{#1}\rfloor}

\newcommand{\lranglet}[2]{\langle{#1},{#2}\rangle}

\newcommand{\eqa}{\stackrel{\rm(a)}{=}}
\newcommand{\eqb}{\stackrel{\rm(b)}{=}}

\DeclareMathOperator*{\argmax}{arg\,max}
\DeclareMathOperator*{\argmin}{arg\,min}

\DeclareMathOperator{\minimize}{minimize}

\DeclareMathOperator{\st}{s.t.}

\DeclareMathOperator{\tr}{tr}

\DeclareMathOperator{\rank}{rank}

\newtheorem{theorem}{Theorem} 
\newtheorem*{theorem*}{Theorem}
\newtheorem{lemma}{Lemma}

\newtheorem{prop}{Proposition}
\newtheorem{corollary}{Corollary}
 
\newtheorem*{assump*}{Assumption}

\theoremstyle{remark}
\newtheorem{remark}{Remark}

\newcommand{\qednew}{\nobreak \ifvmode \relax \else
      \ifdim\lastskip<1.5em \hskip-\lastskip
      \hskip1.5em plus0em minus0.5em \fi \nobreak
      \vrule height0.75em width0.5em depth0.25em\fi}